\documentclass[a4paper,12pt]{amsart}
\usepackage[utf8]{inputenc}
\usepackage[T1]{fontenc}
\usepackage[UKenglish]{babel}
\usepackage[a4paper,margin=28mm]{geometry}
\usepackage{verbatim}

\allowdisplaybreaks[4]
\usepackage{times}
\usepackage{dsfont,mathrsfs}
\usepackage{amsmath}
\usepackage{amsthm}
\usepackage{amssymb}
\usepackage{amsfonts}
\usepackage{latexsym}
\usepackage{booktabs}
\usepackage{mathtools}
 
\usepackage{hyperref}
\hypersetup{hypertex=true,
	colorlinks=true,
	linkcolor=blue,
	anchorcolor=blue,
	citecolor=blue,
	pdftitle={Overleaf Example},
	pdfpagemode=FullScreen}
\usepackage{xcolor}

\newtheorem{theorem}{Theorem}[section]
\newtheorem{lemma}[theorem]{Lemma}
\newtheorem{proposition}[theorem]{Proposition}

\newtheorem{assumption}{Assumption}

\theoremstyle{definition}

\newtheorem{example}[theorem]{Example}

\numberwithin{equation}{section}

\renewcommand{\labelenumi}{\roman{enumi}}
\renewcommand\theenumi\labelenumi
\renewcommand{\leq}{\leqslant}
\renewcommand{\le}{\leqslant}
\renewcommand{\geq}{\geqslant}
\renewcommand{\ge}{\geqslant}

\newcommand{\R}{\mathbb{R}}
\newcommand{\E}{\mathbb{E}}

\newcommand{\dif}{\mathrm{d}}

\newcommand*\abs[1]{\left\lvert#1\right\rvert}

\usepackage{marginnote}
\title[Existence and non-existence of the CLT]{Existence and non-existence of the CLT for a family of SDEs driven by stable processes}

\author[Y. Mo]{Yingjun Mo}
\address[Y.~Mo]{1. Department of Mathematics, Faculty of Science and Technology, University of Macau, Macau, 999078, China; 2. Zhuhai UM Science, Technology Research Institute, Zhuhai, 519031, China}
\email{yc27477@um.edu.mo}

\author[Y. Wang]{Yu Wang}
\address[Y.~Wang]{1. Department of Mathematics, Faculty of Science and Technology, University of Macau, Macau, 999078, China; 2. Zhuhai UM Science, Technology Research Institute, Zhuhai, 519031, China}
\email{yc17447@um.edu.mo}

\begin{document}
	

\begin{abstract}
Stochastic differential equations (SDEs) without global Lipschitz drift often demonstrate unusual phenomena. 
In this paper, we consider the following SDE on $\R^d$: 
 \begin{align*}
    \mathrm{d} \mathbf{X}_t=\mathbf{b}(\mathbf{X}_t) \mathrm{d} t+ \mathrm{d}\mathbf{Z}_t, \quad \mathbf{X}_0=\mathbf{x} \in \mathbb{R}^d,
\end{align*}
where $\mathbf{Z}_t$ is the rotationally symmetric $\alpha$-stable process with $\alpha \in(1,2)$ and $\mathbf{b}:\R^d \rightarrow \R^d$ is a differentiable function satisfying the following condition: there exist some 
$\theta \ge 0$, and $K_1 ,  K_2 ,  L>0$, so that
\begin{align*}
\langle \mathbf{b}(\mathbf{x})-\mathbf{b}(\mathbf{y}), \mathbf{x}-\mathbf{y}\rangle 
\leq
\left\{
	\begin{aligned}
		 \ \ K_1 |\mathbf{x}-\mathbf{y}|^2\ \ , \quad &\forall \ \ |\mathbf{x}-\mathbf{y}| \le L,  \\
		-K_2 |\mathbf{x}-\mathbf{y}|^{2+\theta}, \quad &\forall \ \ |\mathbf{x}-\mathbf{y}| > L.
	\end{aligned}
\right.
\end{align*}
Under this assumption, the SDE admits a unique invariant measure $\mu$. 

We investigate the normal central limit theorem (CLT)  of the empirical measures 
$$
\mathcal{E}_t^\mathbf{x}(\cdot)=\frac{1}{t} \int_0^t \delta_{\mathbf{X}_s }(\cdot) \mathrm{d} s, \ \ \ \ \mathbf{X}_0=\mathbf{x} \in \R^d, \ \ t>0,
$$ 
where $\delta_{\mathbf{x}}(\cdot)$ is the Dirac delta measure.

Our results reveal that, for the bounded measurable function $h$,
$$\sqrt t \left(\mathcal{E}_t^\mathbf{x}(h)-\mu(h)\right)=\frac{1}{\sqrt t} \int_0^t \left(h\left(\mathbf{X}_s^\mathbf{x}\right)-\mu(h)\right) \mathrm{d} s$$ 
admits a normal CLT for $\theta \ge 0$. For the Lipschitz continuous function $h$, the normal CLT does not necessarily hold when $\theta=0$, but it is satisfied for $\theta>1-\frac{\alpha}{2}$.  

\end{abstract}

\maketitle

\section{Introduction}
Stochastic differential equations (SDEs) without global Lipschitz drift often demonstrate unusual phenomena. For instance, \cite{hutzenthaler2011strong,MR3134726,wang2024phase,MR2051587,MR1949404} reported that the Euler-Maruyama (EM) scheme blows up, and \cite{wang2021large} reported a Donsker-Varadhan large deviation for a family of superlinear SDEs driven by $\alpha$-stable processes with $\alpha \in (1,2)$. In this paper, we demonstrate other unusual results which reveal the existence of the normal central limit theorem (CLT) for a family of superlinear SDEs.  

In this paper, we consider the following SDE on $\R^d$: 
 \begin{align}  \label{e:GSDE}
    \mathrm{d} \mathbf{X}_t=\mathbf{b}(\mathbf{X}_t) \mathrm{d} t+ \mathrm{d}\mathbf{Z}_t, \quad \mathbf{X}_0=\mathbf{x} \in \mathbb{R}^d,
\end{align}
where $\mathbf{Z}_t$ is the rotationally symmetric $\alpha$-stable process with $\alpha \in(1,2)$ under a probability base $\left(\Omega, \mathcal{F},\left\{\mathcal{F}_t\right\}_{t \geq 0}, \mathbb{P}\right)$ and the following assumption holds true:  
\begin{assumption} \label{a:Assumption}
$\mathbf{b}:\R^d \rightarrow \R^d$ is a differentiable function satisfying the following condition:
there exist some $\theta \ge 0$, and $K_1,K_2,L>0$, so that
\begin{align*}
\langle \mathbf{b}(\mathbf{x})-\mathbf{b}(\mathbf{y}), \mathbf{x}-\mathbf{y}\rangle 
\leq
\left\{
	\begin{aligned}
		 \ \ K_1 |\mathbf{x}-\mathbf{y}|^2\ \ , \quad &\forall \ \ |\mathbf{x}-\mathbf{y}| \le L,  \\
		-K_2 |\mathbf{x}-\mathbf{y}|^{2+\theta}, \quad &\forall \ \ |\mathbf{x}-\mathbf{y}| > L.
	\end{aligned}
\right.
\end{align*}
\end{assumption}
It is well known that there is a unique weak solution to SDE \eqref{e:GSDE} and the solution is non-explosive for $\theta\geq 0$, see \cite{chen2021supercritical,meyn1993stability,Luo2016RefinedBC} for more details.
It is known (cf, e.g., \cite[Theorem 14.3]{MR3185174}, or \cite{WOS:000343456100011,KIM20142479}) that the L\'evy measure $\nu$ of the process $(\mathbf{Z}_t)_{t \geq 0}$ is
\[
	\nu(\mathrm{d} \mathbf{z}) = \frac{C_{d, \alpha}}{|\mathbf{z}|^{\alpha+d} } \mathrm{d} \mathbf{z} ,\quad\nu(\{\mathbf{0}\})=0,
\]
where the constant $C_{d, \alpha}$ is given by
\begin{align*}
C_{d, \alpha}=\alpha 2^{\alpha-1} \pi^{-d / 2} \frac{\Gamma((d+\alpha) / 2)}{\Gamma(1-\alpha / 2)} ,
\end{align*} 
see e.g, \cite{WOS:000343456100011,MR481057}. By the L\'evy-It\^o decomposition (cf. \cite{MR3185174,applebaum2009levy}), there exists a Poisson random measure $N(\mathrm{ d}{\bf z} , \mathrm{d} t)$ such that
\begin{equation*}
	\mathrm{d} {\bf Z}_t = \int_{|\bf z| > 1} {\bf z} N({\rm d}{\bf z} , {\rm d} t) + \int_{|\bf z| \leqslant 1} {\bf z} \widetilde{N}({\rm d} {\bf z}, {\rm d}t ) ,
\end{equation*}
where $\widetilde{N}({\rm d} {\bf z}, {\rm d}t )  = N({\rm d}{\bf z} , {\rm d} t) - \nu( {\rm d} {\bf z}){\rm d} t$ is the compensated Poisson random measure. We refer to \cite{applebaum2009levy,MR2160585,arapostathis2019uniform,bao2022coupling,baoyuan2017,bao2011comparison,bao2012stochastic,chen2016heat,dong2020irreducibility,KUHN20192654} and references therein for more developments of SDEs driven by L\'evy processes. 

\begin{example}
   Given $\theta \geq 0$, the following SDE on $\mathbb{R}^d$ satisfies Assumption \ref{a:Assumption}:
 \begin{align*}
    \mathrm{d} \mathbf{X}_t=-\mathbf{X}_t\left|\mathbf{X}_t\right|^\theta \mathrm{d} t+ \mathrm{d}\mathbf{Z}_t, \quad \mathbf{X}_0=\mathbf{x} \in \mathbb{R}^d.
\end{align*}
For $\theta=0$, this is a linear SDE, whereas for $\theta>0$, it is a superlinear SDE.
\end{example}

For any $\mathbf{x} \in \mathbb{R}^d$ and $t > 0$, the empirical measure $\mathcal{E}_t^{\mathbf{x}}$ associated with the process $\left(\mathbf{X}_t\right)_{t \geqslant 0}$ in \eqref{e:GSDE} is defined as:
\begin{equation*}
\mathcal{E}_t^\mathbf{x}(A):=\frac{1}{t} \int_0^t \delta_{\mathbf{X}_s^\mathbf{x}}(A) \mathrm{d} s, \quad \forall A \in \mathcal{B}\left(\mathbb{R}^d\right),
\end{equation*}
where $\mathcal{B}\left(\mathbb{R}^d\right)$ denotes the collection of Borel sets on $\mathbb{R}^d$, $\delta_{\mathbf{y}}(\cdot)$ is the Dirac delta measure defined as $\delta_{\mathbf{y}}(A)=1$ if $\mathbf y \in A$ and $\delta_{\mathbf{y}}(A)=0$ otherwise.  Note that we write $\mathbf{X}_t^\mathbf{x}$ instead of $\mathbf{X}_t$ to emphasize the starting point $\mathbf{X}_0=\mathbf{x}$ for a given $\mathbf{x} \in \mathbb{R}^d$. For any measurable test function $h: \mathbb{R}^d \rightarrow \mathbb{R}$, the empirical measure satisfies:
\begin{equation}\label{empirical}
\mathcal{E}_t^\mathbf{x}(h)=\frac{1}{t} \int_0^t h\left(\mathbf{X}_s^\mathbf{x}\right) \mathrm{d} s .
\end{equation}

In this paper, we examine the existence and non-existence of the CLT for the empirical measures $\mathcal{E}_t^{\mathbf{x}}(h)$ of the SDE \eqref{e:GSDE} under different values of the parameter $\theta$ in Assumption \ref{a:Assumption}, using bounded measurable or Lipschitz continuous test functions. As summarized in Table \ref{table1}, our results reveal that: for the bounded measurable function $h$, the empirical measure $\mathcal{E}_t^\mathbf{x}(h)$ admits a normal CLT for any $\theta \ge 0$; in contrast, for the Lipschitz continuous function $h$, different results emerge. When $\theta = 0$, the empirical measure $\mathcal{E}_t^\mathbf{x}(h)$ does not necessarily satisfy the CLT (non-CLT). However, when $\theta >1-\frac{\alpha}{2}$ , the CLT is indeed satisfied. This observation underscores the pivotal role of the parameter $\theta$ in governing the asymptotic behavior of empirical measures, particularly with respect to the existence or non-existence of the CLT.
\begin{table}[ht]{\label{table1}}
\centering
\caption{Existence and Non-existence of Normal CLT for SDE \eqref{e:GSDE}.}
\begin{tabular}{|c|c|c|c|}
\hline
Test function $h$ & $\theta = 0$ & $0<\theta \leq 1-\frac{\alpha}{2}$ & $\theta > 1-\frac{\alpha}{2}$ \\
\hline
Bounded measurable& Normal CLT & Normal CLT & Normal CLT \\
\hline
Lipschitz continuous&No Normal CLT & Unknown & Normal CLT \\
\hline
\end{tabular}
\end{table}

There has been considerable research on the central limit problem for functionals of solutions to SDEs.  For SDEs driven by Brownian motion, several results have been established (see, for example, \cite{bao2024limit,suo2021central,lu2022central,fang2019multivariate}). In particular, Bao et al. \cite{bao2024limit} established the asymptotic normality of integrals with the form $t^{-\frac{1}{2}}\int_{0}^{t}h(\mathbf X_{s}^{\mathbf{x}})\mathrm{d}s$ for one class of continuous functions $h$, where $\mathbf X_{s}^{\bf x}$ solves the Brownian SDE with dissipative and Hölder continuous drift. Bao et al. \cite{MR4128304} established CLT for additive functionals of path-dependent SDEs.

For the central limit problem related to Lévy processes, there are also significant results, such as those in \cite{qiao2022limit,dereich2016multilevel,xu2019approximation,li2023stable}, which primarily focus on the scaled average of independent random variables. In particular, \cite{jin2024approximation} established the CLT for the empirical measure of piecewise $\alpha$ stable Ornstein-Uhlenbeck processes with a bounded measurable test function. In this paper, we investigate the central limit problem for SDEs driven by rotationally invariant $\alpha$-stable Lévy processes, our findings reveal that for SDEs with linear and superlinear drift terms, under Lipschitz continuous test functions $h$, the limiting behaviors of the empirical measures $\mathcal{E}_t^\mathbf{x}(h)$ are markedly different.

Furthermore, research on the non-CLT related to Lévy processes remains relatively sparse. \cite{MR2362589} established a functional non-CLT for jump-diffusions with periodic coefficients driven by stable Lévy processes with $\alpha \in (1,2)$, where the limiting process is shown to be an $\alpha$-stable Lévy process with an averaged jump measure. This paper demonstrates that for SDEs driven by stable Lévy processes, under specific conditions on both the drift and diffusion terms, the empirical measures $\mathcal{E}_t^{\mathbf{x}}(h)$ converge to a stable distribution, as shown in proposition \ref{thm:Limit}, which also constitutes a significant contribution to the study.

\subsection{Notations}
Throughout this paper, we use the following notations. Let $\mathbb{R}^d$ (and $\mathbb{R}_{+}^d$ ), $d \geqslant 1$, be real (nonnegative) valued $d$-dimensional vectors, and write $\mathbb{R}$ (and $\mathbb{R}_{+}$) for the case when $d=1$. For any $\mathbf{x}, \mathbf y \in \mathbb{R}^d,\langle \mathbf{x}, \mathbf y\rangle$ means the inner product, that is, $\langle \mathbf{x}, \mathbf y\rangle=\mathbf{x}^{\mathrm{T}} \mathbf y$ where $\mathbf{x}^{\mathrm{T}}$ means the transpose for $\mathbf{x}$. This reduces the Euclidean metric denoted by $|\cdot|$. Given a matrix $\mathbf A \in \mathbb{R}^{d \times d}$, its Hilbert-Schmidt norm is $\|\mathbf{A}\|_{\mathrm{HS}}=\sqrt{\sum_{i, j=1}^d \mathbf A_{i j}^2}=\sqrt{\operatorname{Tr}\left(\mathbf A^{\mathrm{T}} \mathbf A\right)}$ and its operator norm is $\|\mathbf A\|_{\mathrm{op}}=\sup _{|\mathbf{v}|=1}|\mathbf A \mathbf{v}|$.

Let $\mathcal{B}\left(\mathbb{R}^d, \mathbb{R}\right)$ be the collection of all $\mathbb{R}$-valued Borel measurable functions defined on $\mathbb{R}^d$, and $\mathcal{B}_b\left(\mathbb{R}^d, \mathbb{R}\right)$ be the collection of all $\mathbb{R}$-valued Borel bounded measurable functions defined on $\mathbb{R}^d$. Let $\mathcal{C}^1\left(\mathbb{R}^d, \mathbb{R}\right)$ be the collection of all $\mathbb{R}$-valued differential continuous functions defined on $\mathbb{R}^d$, and $\mathcal{C}_b\left(\mathbb{R}^d, \mathbb{R}\right)$ be the collection of all $\mathbb{R}$-valued bounded measurable functions defined on $\mathbb{R}^d$. For $f \in \mathcal{C}_b\left(\mathbb{R}^d, \mathbb{R}\right)$, denote $\|f\|_{\infty}:=\sup _{\mathbf{x} \in \mathbb{R}^d}|f(\mathbf{x})|$. For $f \in \mathcal{C}^1\left(\mathbb{R}^d, \mathbb{R}\right)$, $\nabla f \in \mathbb{R}^d$ is the gradient of $f$, and we denote 
\begin{align*}
     \|\nabla  f\|_{\infty} := \sup_{\mathbf{x},\mathbf v \in \R^d} \left\{ \langle \nabla  f (\mathbf{x}),\mathbf v \rangle \colon \mathbf{x}, \mathbf v \in \R^d; \; \abs{\mathbf v} \leq 1 \right\}.
\end{align*}

Let $\mathcal{M}_1\left(\mathbb{R}^d\right)$ be the space of probability measures on $\mathbb{R}^d$, equipped with the  Borel $\sigma$-field $\mathcal{B}\left(\mathbb{R}^d \right)$. For any probability measure $\mu \in \mathcal{M}_1\left(\mathbb{R}^d\right)$, denote $\mu(f)=\int_{\mathbb{R}^d} f(\mathbf{x}) \mu(\mathrm{d}\mathbf{x})$. Let $\mathscr{C}\left(\mu_1, \mu_2\right)$ denote the class of couplings of probability measures $\mu_1, \mu_2\in \mathcal{M}_1\left(\mathbb{R}^d\right)$. The $L^1$-Wasserstein distance between $\mu_1$ and $\mu_2$, denoted by $\mathbb{W}_1\left(\mu_1, \mu_2\right)$, is defined as, 
$$
\mathbb{W}_1\left(\mu_1, \mu_2\right):=\inf _{\Pi \in \mathscr{C}\left(\mu_1, \mu_2\right)}\left(\int_{\mathbb{R}^d \times \mathbb{R}^d} |\mathbf{x}-\mathbf y| \Pi(\mathrm{d}\mathbf{x}, \mathrm{d}\mathbf y)\right) .
$$
By Kantorovich's dual formula, we have
\begin{align} \label{Kantorovich}
\mathbb{W}_1\left(\mu_1, \mu_2\right)=\sup _{f\in \mathrm{Lip}(1)}\left|\mu_1(f)-\mu_2(f)\right|
\end{align}
where $\operatorname{Lip}(1)=\left\{f: \mathbb{R}^d \rightarrow \mathbb{R} ;|f(\mathbf{x})-f(\mathbf y)| \leq|\mathbf{x}-\mathbf y|,\ \forall \ {\bf x}, {\bf y} \in \mathbb{R}^d \right\}$. 
The total variation distance of measures $\mu_1$ and $\mu_2$, denoted by $\left\|\mu_1-\mu_2\right\|_{\mathrm{TV}}$, is defined as, 
$$
\left\|\mu_1-\mu_2\right\|_{\mathrm{TV}}=2 \mathbb{W}_0\left(\mu_1, \mu_2\right):=\int_{\mathbb{R}^d}\left|\mu_1-\mu_2\right| \mathrm{d}\mathbf{x} =\sup _{\|f\|_{\infty} \leq 1}\left|\mu_1(f)-\mu_2(f)\right|.
$$

A sequence of random variables $\left\{Y_n, n \geqslant 1\right\}$ is said to converge weakly or converge in distribution to a limit $Y_{\infty}$, denoted by $Y_n \Rightarrow Y_{\infty}$, if their distribution functions converge weakly. In addition, $Y_n \xrightarrow{\mathrm{P}} Y_{\infty}$ means convergence in probability if 
\begin{align*}
    \lim _{n \rightarrow \infty} \mathbb{P}\left(\left|Y_n-Y_{\infty}\right|>\varepsilon\right)=0, \quad \text{ for all } \varepsilon>0.
\end{align*}

The associated Markov semigroup $(\mathrm{P}_t)_{t \geqslant 0}$ of the SDE \eqref{e:GSDE} is given by
$$
\mathrm{P}_t f(\mathbf{x})=\mathbb{E} f\left(\mathbf{X}_t^\mathbf{x}\right), \quad \forall\ \mathbf{x} \in \mathbb{R}^d, \quad f \in \mathcal{B}_b\left(\mathbb{R}^d, \mathbb{R}\right),\quad t\geq 0.
$$
The infinitesimal generator $\mathcal{A}$ of the process $(\mathbf{X}^{\mathbf{x}}_t)_{t \geqslant 0}$ is defined as follows: for $f \in \mathcal{D}(\mathcal{A})$,  
\begin{align}\label{generator}  
\mathcal{A} f(\mathbf{x})=\left\langle \mathbf{b}(\mathbf{x}), \nabla f(\mathbf{x})\right\rangle + \int_{\mathbb{R}_0^d}\left[f(\mathbf{x}+\mathbf z)-f(\mathbf{x})-\langle\mathbf  z, \nabla f(\mathbf{x})\rangle \mathbf{ 1}_{(0,1)}(|\mathbf z|)\right] \nu(\mathrm{d} \mathbf z),
\end{align}
where $\mathbb{R}_0^d=\mathbb{R}^d \backslash\{0\}$, and $\mathcal{D}(\mathcal{A})$ is the domain of $\mathcal{A}$, which depends on the function space where the semigroup $(\mathrm{P}_t)_{t \geqslant 0}$ operates. 

For a time homogeneous Markov process  $(\mathbf X_t^\mathbf{x})_{t\geq 0}$ on $\R^d$, it is ergodic \cite{borovkov2001piece} if there exists a unique invariant probability $\pi$ such that  
\begin{equation*}
\lim_{t\to \infty} \| \mathcal{L}(\mathbf X^\mathbf{x}_t) - \pi\|_{\rm TV} = 0, \quad \mathbf{x}\in \R^d ,
\end{equation*}
 where $\mathcal{L}({\bf X}_t^{\bf x})$ is the distribution of ${\bf X}_t^{\bf x}$. 

 The ergodicity can be further classified as the following categories: 

 (i) \emph{$V$-exponential ergodicity}: There exists some continuous function $V: \R^d \rightarrow [1,\infty)$ such that there exist some positive constants $c,C$ satisfying
\begin{eqnarray*}
\sup _{|f| \leq V }\left|\mathbb{E}\left[f\left(\mathbf X_t^\mathbf{x}\right)\right]-\pi(f)\right|
&\leq& CV(\mathbf{x})e^{-ct}, \quad \mathbf{x}\in \R^d.
\end{eqnarray*}

(ii) \emph{$V$-uniform ergodicity}: There exists some continuous function $V: \R^d \rightarrow [1,\infty)$ such that there exist some positive constants $c,C$ satisfying
\begin{eqnarray*}
\sup _{|f| \leq V }\left|\mathbb{E}\left[f\left(\mathbf X_t^\mathbf{x}\right)\right]-\pi(f)\right|
&\leq& Ce^{-ct}, \quad \mathbf{x}\in \R^d.
\end{eqnarray*}

It can be shown that $\left(\mathbf X^{\mathbf{x}}_t\right)_{t \geqslant 0}$ in \eqref{e:GSDE} is exponentially ergodic, we denote $\mu$ by its invariant measure, and more details can be found in Lemma \ref{LemmaErgodicity1}. For any Lipschitz continuous test function $h$, it is easy to know that $\mu(h)<\infty$ by Lemma \ref{LemmaErgodicity2}. Let $f_h$ be the solution to the following Poisson equation:  
\begin{equation} \label{Poisson}
\mathcal{A} f_h(\mathbf{x})=h(\mathbf{x})-\mu(h),
\end{equation}
where $\mathcal{A}$ is defined by \eqref{generator} and $h$ is a bounded measurable or Lipschitz continuous test function.

\subsection{Main Results}{\label{Results}}

 The main theorems of this paper are presented below, whose proofs are provided in Section \ref{ProofResults}. We define $\mathcal{V}\left(f_h\right)$ as the following:
\begin{equation}
	\label{e:Vh}
	\mathcal{V}\left(f_h\right)=\int_{\mathbb{R}^d} \int_{\mathbb{R}_0^d}[f_h(\mathbf{x}+\mathbf z)-f_h(\mathbf{x})]^2 \nu(\mathrm{d} \mathbf z) \mu(\mathrm{d} \mathbf{x}),
\end{equation}
and let  $\mathcal{N}\left(0, \mathcal{V}\left(f_h\right)\right)$ be the center Gaussian distribution with variance $\mathcal{V}\left(f_h\right)$.

When the test function $h:\R^d \rightarrow \R$
is a bounded measurable function, the following theorem yields that $\sqrt{t}\left[\mathcal{E}_t^{\mathbf{x}}(h)-\mu(h)\right]$ admits a normal CLT. 

\begin{theorem}\label{thm:Limit1}
Let Assumption \ref{a:Assumption} hold. Consider the SDE \eqref{e:GSDE} and the empirical measure $\mathcal{E}_t^\mathbf{x}(h)$ defined in \eqref{empirical}. For all $\theta \ge 0$ and for any bounded measurable function $h:\R^d \rightarrow \R$, the term $\sqrt{t}\left[\mathcal{E}_t^{\mathbf{x}}(h)-\mu(h)\right]$ converges weakly to the Gaussian distribution $\mathcal{N}\left(0, \mathcal{V}\left(f_h\right)\right)$ as $t \rightarrow \infty$,  where $\mathcal{V}(f_n)$ is defined in \eqref{e:Vh}.
\end{theorem}
For a Lipschitz continuous test function $h$, the normal CLT is satisfied for $\theta>1-\frac{\alpha}{2}$, but this theorem does not necessarily hold when $\theta=0$.
\begin{theorem}\label{thm:Limit2}
Let Assumption \ref{a:Assumption} hold. Consider the SDE \eqref{e:GSDE} and the empirical measure $\mathcal{E}_t^\mathbf{x}(h)$ defined in \eqref{empirical}. For any Lipschitz continuous test function $h:\R^d \rightarrow \R$, we have

(i) For the case that $\theta=0$, the term $\sqrt{t}\left[\mathcal{E}_t^{\mathbf{x}}(h)-\mu(h)\right]$ does not necessarily converge weakly to a Gaussian distribution as $t \rightarrow \infty$, that is, the CLT does not necessarily hold.

(ii) For the case that $\theta>1-\frac{\alpha}{2}$, the term $\sqrt{t}\left[\mathcal{E}_t^{\mathbf{x}}(h)-\mu(h)\right]$ converges weakly to the Gaussian distribution $\mathcal{N}\left(0, \mathcal{V}\left(f_h\right)\right)$ as $t \rightarrow \infty$,  where $\mathcal{V}(f_n)$ is defined in \eqref{e:Vh} .
\end{theorem}

\section*{Acknowledgement}
We are deeply grateful to our supervisor, Lihu Xu, for his patience in discussing this research with us and for his meticulous revisions of the early drafts of this paper. We also sincerely appreciate Xinghu Jin for his valuable suggestions.

\section{Auxiliary Lemmas}
In this section, we mainly present the conclusions on the ergodicity and contractivity of the solution $(\mathbf X^\mathbf{x}_{t})_{t\ge0}$ of SDE \eqref{e:GSDE}.

By Luo et al. \cite[Theorem 1.1 and Proposition 1.5]{Luo2016RefinedBC}, we have the following ergodic-type result, which shows that the solution $(\mathbf X^\mathbf{x}_{t})_{t\ge0}$ of SDE \eqref{e:GSDE} admits a unique invariant probability measure $\mu$. Furthermore, the law of $\mathbf X^\mathbf{x}_{t}$ converges to $\mu$ at a exponential rate. 

\begin{lemma}[Ergodicity for the SDE \eqref{e:GSDE}]\label{LemmaErgodicity1}
Denote by $\left(\mathbf X_t^\mathbf{x}\right)_{t \geq 0}$ the solution to the SDE \eqref{e:GSDE}. 

(i) For the case $\theta=0$, $\left(\mathbf X^\mathbf{x}_t\right)_{t \geq 0}$ admits a unique invariant probability measure $\mu$. Additionally, there exist some positive constants $c_1(\mathbf{x})$ (depending on the initial point ${\bf x}$) and $c_2$, such that 
\begin{equation*}
\| \mathcal{L}(\mathbf X^\mathbf{x}_t) - \mu\|_{\rm TV} \leq c_1(\mathbf{x}) {\rm e}^{-c_2 t}, \quad \mathbf{x} \in \mathbb{R}^d, \quad t>0.
\end{equation*}

(ii) For the case $\theta>0$, $\left(\mathbf X^\mathbf{x}_t\right)_{t \geq 0}$ admits a unique invariant probability measure $\mu$. Furthermore, there exists some uniformly positive constants $c_1, c_2$ that are independent of ${\bf x}$, such that 
\begin{equation*}
\| \mathcal{L}(\mathbf X^\mathbf{x}_t) - \mu\|_{\rm TV} \leq c_1 {\rm e}^{-c_2 t}, \quad \mathbf{x} \in \mathbb{R}^d, \quad t>0.
\end{equation*}
\end{lemma}

{We define a function $V_{\beta}(\mathbf{x}): \mathbb{R}^d \to \mathbb{R}$ as:}
\begin{equation}
	\label{LyapunovFunction}
	V_{\beta}(\mathbf{x})=(1+\vert \mathbf{x}\vert^{2})^{{\beta}/{2}},\quad  \mathbf{x}\in\mathbb{R}^{d},\quad  \beta \in[1,\alpha).
\end{equation}

By applying Lemma \ref{LemmaErgodicity1}, we obtain the following result, which is related to $V_\beta$-exponential ergodicity and the moments estimation of $\left(\mathbf X^{\mathbf{x}}_t\right)_{t \geqslant 0}$ in \eqref{e:GSDE}.

\begin{lemma} \label{LemmaErgodicity0}
Denote by $\left(\mathbf X_t^\mathbf{x}\right)_{t \geq 0}$ the solution to the SDE \eqref{e:GSDE} for the case $\theta\geq 0$. $\left(\mathbf X_t^\mathbf{x}\right)_{t \geq 0}$ admits a unique invariant probability measure $\mu$ such that, for $1 \leq \beta<\alpha$
\begin{align}\label{LemmaErgodicity21}
\sup _{|f| \leq V_\beta}\left|\mathbb{E}\left[f\left(\mathbf X_t^\mathbf{x}\right)\right]-\mu(f)\right| &\leq c_1 V_\beta(\mathbf{x}) \mathrm{e}^{-c_2 t}, \quad t>0,
\end{align}
for some constants $c_1, c_2>0$. In particular, there exists a constant $C>0$ such that
\begin{align}\label{EXtbeta}
\mathbb{E}\left|\mathbf X_t^\mathbf{x}\right|^\beta \leq C\left(1+|\mathbf{x}|^\beta\right), \quad t>0.
\end{align}
\end{lemma}
\begin{proof}
    (i) By the definition of $V_{\beta}(\mathbf{x})$ in \eqref{LyapunovFunction}, it is easy to check that $V_\beta \in \mathcal{D}\left(\mathcal{A}\right)$. Recalling the generator $\mathcal{A}$ of the process $\mathbf X_t^\mathbf{x}$ from \eqref{generator}, we have that 
\begin{equation*}
	\begin{aligned}
		 \mathcal{A} V_{\beta}(\mathbf{x})  
		= &\langle \mathbf b(\mathbf{x}), \nabla V_{\beta}(\mathbf{x})\rangle + C_{d, \alpha} \int_{\mathbb{R}_0^d}\left( V_{\beta}(\mathbf{x}+\mathbf z) - V_{\beta}(\mathbf{x}) - \left\langle\nabla V_{\beta}(\mathbf{x}), \mathbf z\right\rangle  \mathbf{1}_{(0, 1)}(|\mathbf z|)  \right) \frac{\mathrm{d} \mathbf z}{|\mathbf z|^{\alpha+d}} \\
		=:& H_1 + H_2.  
\end{aligned}
\end{equation*}
Next, we estimate $H_1$ and $H_2$ respectively.

For $V_{\beta}(\mathbf{x})$, we have
$$
\nabla V_\beta(\mathbf{x})=\frac{\beta \mathbf{x}}{\left(1+|\mathbf{x}|^2\right)^{1-\frac{\beta}{2}}}, \quad \nabla^2 V_\beta(\mathbf{x})=\frac{\beta {\bf I}_d}{\left(1+|\mathbf{x}|^2\right)^{1-\frac{\beta}{2}}}+\frac{\beta(\beta-2) \mathbf{x} \mathbf{x}^{\mathrm{T}}}{\left(1+|\mathbf{x}|^2\right)^{2-\frac{\beta}{2}}},
$$
where ${\bf I}_d$ denotes the $d \times d$ identity matrix. We see that for any $\mathbf{x} \in \mathbb{R}^d$,
\begin{align}\label{nablaV}
\left|\nabla V_\beta(\mathbf{x})\right| \leq \beta|\mathbf{x}|^{\beta-1}, \quad\left\|\nabla^2 V_\beta(\mathbf{x})\right\|_{\mathrm{HS}} \leq \beta(3-\beta) \sqrt{d} .
\end{align}

For the term $H_1$, we have
\begin{equation} \label{e:I}
	H_1 = \left\langle \mathbf b(\mathbf{x}), \nabla V_\beta(\mathbf{x})\right\rangle  =\frac{\beta}{\left(1+|\mathbf{x}|^2\right)^{\frac{2-\beta}{2}}}\langle \mathbf b(\mathbf{x})-\mathbf b(\mathbf{0}), \mathbf{x}\rangle + \frac{\beta}{\left(1+|\mathbf{x}|^2\right)^{\frac{2-\beta}{2}}}\langle \mathbf b(\mathbf{0}), \mathbf{x}\rangle. 
\end{equation}
For the first term of \eqref{e:I}. By Assumption \ref{a:Assumption} and the fact that $|\mathbf{x}|^{2+\theta}\geq (\frac{1+|\mathbf{x}|^2}{2})^{1+\frac{\theta}{2}}$ for $|\mathbf{x}|\geq 1$, we have
\begin{align*}
&\mathrel{\phantom{=}} \frac{\beta}{\left(1+|\mathbf{x}|^2\right)^{\frac{2-\beta}{2}}}\langle \mathbf b(\mathbf{x})-\mathbf b(\mathbf{0}), \mathbf{x}\rangle \\
& \leq \frac{\beta|\mathbf{x}|^2}{\left(1+|\mathbf{x}|^2\right)^{\frac{2-\beta}{2}}}(-K_2|\mathbf{x}|^{\theta}\mathbf 1_{\{|\mathbf{x}|>L\}}+K_1 \mathbf 1_{\{|\mathbf{x}|\leq L\}}) \\
& \leq -K_2\frac{\beta|\mathbf{x}|^{2+\theta}}{\left(1+|\mathbf{x}|^2\right)^{\frac{2-\beta}{2}}} \mathbf 1_{\{|\mathbf{x}|> 1\}}+\beta(K_2L^{2+\theta}+K_1L^2) \\
& \leq -2^{-(1+\frac{\theta}{2})} K_2 \beta {\left(1+|\mathbf{x}|^2\right)^{\frac{\beta+\theta}{2}}} \mathbf 1_{\{|\mathbf{x}|> 1\}}+\beta(K_2L^{2+\theta}+K_1L^2) \\
& \leq -2^{-(1+\frac{\theta}{2})}K_2 \beta {\left(1+|\mathbf{x}|^2\right)^{\frac{\beta+\theta}{2}}}+\beta\left( K_2  { 2 ^{\frac{\beta}{2}-1}}+K_2L^{2+\theta}+K_1L^2\right).
\end{align*}
For the second term of \eqref{e:I}. By Young's inequality, it holds that
\begin{align*}
	&\frac{\beta}{\left(1+|\mathbf{x}|^2\right)^{\frac{2-\beta}{2}}}\langle \mathbf b(\mathbf{0}), \mathbf{x}\rangle 
	\leq \beta|\mathbf b(\mathbf{0})||\mathbf{x}|^{\beta-1}\\
    \leq& |\mathbf b(\mathbf{0})|^\beta+(\beta-1)V_\beta(\mathbf{x})\\
\leq& |\mathbf b(\mathbf{0})|^\beta+( \beta-1)\left[2^{-(1+\frac{\theta}{2})}K_2V_\beta(\mathbf{x})^{1+\frac{\theta}{\beta}}+  2^{\beta(\frac{1}{\theta}+\frac{1}{2})} K_2^{-\frac{\beta}{\theta}}\right].
\end{align*}
Thus, substituting above two inequalities into \eqref{e:I} leads to
\begin{equation}
	\label{e:bound_I}
	\begin{aligned}
		H_1 \leq-2^{-(1+\frac{\theta}{2})}K_2  V_\beta(\mathbf{x})^{1+\frac{\theta}{\beta}}+C.
	\end{aligned} 
\end{equation}

For the part $H_2$, we notice that
\begin{equation}\label{e:decomp_J}
	\begin{aligned}
		&\mathrel{\phantom{=}}
		\int_{\mathbb{R}_0^d}\left(V_\beta(\mathbf{x}+\mathbf z)-V_\beta(\mathbf{x})-\left\langle\nabla V_\beta(\mathbf{x}), \mathbf z\right\rangle \mathbf{1}_{(0, 1)}(|\mathbf z|)\right) \frac{\mathrm{d} \mathbf z}{|\mathbf z|^{\alpha+d}} \\
		&= \int_{|\mathbf z|\leq 1} \int_0^1 \int_0^r\left\langle\nabla^2 V_\beta(\mathbf{x}+s \mathbf z), \mathbf z \mathbf z^{\mathrm{T}}\right\rangle_{\mathrm{HS}} \mathrm{d} s \mathrm{d} r \frac{\mathrm{d} \mathbf z}{|\mathbf z|^{\alpha+d}} \\
		&\quad
		 + \int_{|\mathbf z| > 1} \int_0^1\left\langle\nabla V_\beta(\mathbf{x}+r \mathbf z), \mathbf z\right\rangle \mathrm{d} r \frac{\mathrm{ d} \mathbf z}{|\mathbf z|^{\alpha+d}} \\
		 &=: H_{21} + H_{22}.
	\end{aligned}
\end{equation}
By \eqref{nablaV}, we have
\begin{equation}
	\label{e:bound_J1}
	\begin{aligned}
		H_{21}
		&\leq \frac{\beta(3-\beta)}{2} \sqrt{d} \int_{|\mathbf z|\leq 1} \frac{|\mathbf z|^2}{|\mathbf z|^{\alpha+d}} \mathrm{ d} \mathbf z
		\leq \frac{\beta(3-\beta) \sqrt{d} \sigma_{d-1}}{2(2-\alpha)} ,
	\end{aligned}
\end{equation}
and
\begin{equation}
	\label{e:bound_J2}
	\begin{aligned}
		H_{22} 
		\leq \beta \int_{|\mathbf z| > 1} \frac{|\mathbf{x}|^{\beta-1}|\mathbf z|+|\mathbf z|^\beta}{|\mathbf z|^{\alpha+d}} \mathrm{d}\mathbf z 
		\leq \beta \sigma_{d-1}\left(\frac{|\mathbf{x}|^{\beta-1}}{\alpha-1}+\frac{1}{\alpha-\beta}\right),
	\end{aligned}
\end{equation}
where $\sigma_{d-1}=2 \pi^{\frac{d}{2}} / \Gamma\left(\frac{d}{2}\right)$ is the surface area of the unit sphere $\mathbb{S}^{d-1} \subset \mathbb{R}^d$. 

Combining \eqref{e:decomp_J}, \eqref{e:bound_J1} and \eqref{e:bound_J2}, we have
\begin{equation}\label{e:bound_J}
	\begin{aligned}
		H_2 =& C_{d, \alpha} \int_{\mathbb{R}_0^d}\left( V_{\beta}(\mathbf{x}+\mathbf z) - V_{\beta}(\mathbf{x}) - \left\langle\nabla V_{\beta}(\mathbf{x}), \mathbf z\right\rangle \mathbf{1}_{(0, 1)}(|\mathbf z|) \right) \frac{\mathrm{d} \mathbf z}{|\mathbf z|^{\alpha+d}} \\
		\leq & \frac{ C_{d, \alpha} \beta(3-\beta) \sqrt{d} \sigma_{d-1}}{2(2-\alpha)}
		+ C_{d, \alpha} \beta \sigma_{d-1}\left(\frac{|\mathbf{x}|^{\beta-1}}{\alpha-1}+\frac{1}{\alpha-\beta}\right) \\
        \leq&\frac{C_{d, \alpha} \beta(3-\beta) \sqrt{d} \sigma_{d-1}}{2(2-\alpha)}+\frac{C_{d, \alpha} \beta \sigma_{d-1}}{\alpha-\beta}+\left(\frac{C_{d, \alpha}\beta \sigma_{d-1}}{\alpha-1}\right)^\beta+ V_\beta(\mathbf{x})\\
\leq&\frac{C_{d, \alpha} \beta(3-\beta) \sqrt{d} \sigma_{d-1}}{2(2-\alpha)}+\frac{C_{d, \alpha} \beta \sigma_{d-1}}{\alpha-\beta}+\left(\frac{C_{d, \alpha} \beta\sigma_{d-1}}{\alpha-1}\right)^\beta\\
&+2^{\beta(\frac{2}{\theta}+\frac{1}{2})} K_2^{-\frac{\beta}{\theta}}+2^{-(2+\frac{\theta}{2})}K_2 V_\beta(\mathbf{x})^{1+\frac{\theta}{\beta}},
	\end{aligned}
\end{equation}
where the last two inequalities come from Young's inequality.

Hence, combining \eqref{e:bound_I}, \eqref{e:bound_J}, we have
\begin{align}\label{Vbeta}
\mathcal{A} V_\beta( \mathbf{x} ) \leq- C_1V_\beta(\mathbf{x})^{1+\frac{\theta}{\beta}}+C_2,
\end{align}
where $C_1=2^{-(2+\frac{\theta}{2})}K_2$.  
Since $V_{\beta}(\bf x)\geq 1$, we also have 
\begin{equation}
	\label{e:dif_V}
	\mathcal{A} V_\beta( \mathbf{x} )\leq- C_1V_\beta(\mathbf{x}) +C_2.
\end{equation}
By Lemma \ref{LemmaErgodicity1}, we know that $\left(\mathbf X^\mathbf{x}_t\right)_{t \geq 0}$ admits a unique invariant probability measure $\mu$ for any $\theta\geq 0$. Then, by \cite[Theorem 6.1]{meyn1993stability}, we obtain the first result \eqref{LemmaErgodicity21}, that is,
\begin{align*}
\sup _{|f| \leq V_\beta}\left|\mathbb{E}\left[f\left(\mathbf X_t^\mathbf{x}\right)\right]-\mu(f)\right| \leq c_1 V_\beta(\mathbf{x}) \mathrm{e}^{-c_2 t}.
\end{align*}

(ii) By using Itô's formula, one has 
$$\mathbb{E} V_\beta\left(\mathbf X^\mathbf{x}_t\right)=V_\beta(\mathbf{x})+\mathbb{E} \int_0^t \mathcal{A} V_\beta\left(\mathbf X_s^\mathbf{x}\right) \mathrm{d} s$$
which, together with \eqref{e:dif_V}, implies
\begin{align}\label{Vbeta1}
\frac{\mathrm{d}}{\mathrm{d} t} \mathbb{E} V_\beta\left(\mathbf X^\mathbf{x}_t\right) 
\leqslant-C_1 \mathbb{E}V_\beta\left(\mathbf X_t^\mathbf{x}\right) + C_2.
\end{align}
By the Gr\"onwall's inequality, we can obtain that
$$
\mathbb{E} V_\beta\left(\mathbf{X}_t^{\mathbf{x}}\right) \leqslant V_\beta(\mathbf{x}) \mathrm{e}^{-C_1 t}+C_2 / C_1.
$$
Due to $|\mathbf{x}|^\beta \leqslant V_\beta(\mathbf{x}) \leq C(1+{|\mathbf{x}|}^{\beta})$ for all $\mathbf{x} \in \mathbb{R}^d$, we have the desired result \eqref{EXtbeta}. 
\end{proof}

Moreover, the following $L^1$-Wasserstein distance for the SDE \eqref{e:GSDE} can be established. For further details, refer to \cite[Theorem 1.2]{wang2016p}. 
\begin{lemma}[$L^1$-Wasserstein distance for SDE for the case $\theta\geq 0$]{\label{contract}}
Denote by $\left(\mathbf X_t\right)_{t \geq 0}$ the solution to the SDE \eqref{e:GSDE} for the case $\theta\geq 0$. Then, there exist constants $\lambda, C>0$ such that 

(i) If $\theta=0$, for every $t>0$, and $\mathbf{x}, \mathbf y \in \mathbb{R}^d $,  
\begin{align*}
\mathbb W_1\left(\mathcal{L}\left(\mathbf X_t^\mathbf{x}\right), \mathcal{L}\left(\mathbf X_t^\mathbf y\right)\right) \leq C \mathrm{e}^{-\lambda t} |\mathbf{x}-\mathbf y|.
\end{align*}

(ii) If $\theta>0$, for every $t>0$, and $\mathbf{x}, \mathbf y \in \mathbb{R}^d $, 
\begin{align*}
\mathbb W_1\left(\mathcal{L}\left(\mathbf X_t^\mathbf{x}\right), \mathcal{L}\left(\mathbf X_t^\mathbf y\right)\right) \leq C \mathrm{e}^{-\lambda t}\left[\frac{|\mathbf{x}-\mathbf y|}{1+|\mathbf{x}-\mathbf y| \mathbf{1}_{(1, \infty)}(t)}\right].
\end{align*}
\end{lemma}
Lemmas \ref{LemmaErgodicity0} and \ref{contract} imply the uniform ergodicity for the SDE \eqref{e:GSDE} with $\theta>0$ under $L^1$-Wasserstein distance, i.e.
\begin{lemma}[Uniform ergodicity for the SDE \eqref{e:GSDE} with $\theta>0$] \label{LemmaErgodicity2}
Denote by $\left(\mathbf X_t\right)_{t \geq 0}$ the solution to the SDE \eqref{e:GSDE} for the case $\theta> 0$.  
Then, there exists a constant $C>0$ independent of initial point $\mathbf{x}$, such that for $1 \leq \beta<\alpha$ and $\mathbf{x} \in \mathbb{R}^d$,  
\begin{align}\label{e:Lem2.4_1}
\sup _{t\geq 1}\mathbb{E}V_{\beta}(\mathbf X_t^\mathbf{x}) \leq C,
\end{align}
and
\begin{align}\label{e:Lem2.4_2}
	\int_0^1\mathbb{E}V_{\beta}(\mathbf X_t^\mathbf{x}) \mathrm{d} t \leq C\left[(1+|\mathbf{x}|^2)^{\frac{\beta- \theta }{2}}+1\right].
\end{align}
Moreover, $\left(\mathbf X_t\right)_{t \geq 0}$ admits a unique invariant probability measure $\mu$ such that, $\mu(V_\beta)<\infty$, and for some constants $c_1, c_2>0$,
\begin{align}\label{LemmaErgodicity22}
\mathbb W_1\left(\mathcal{L}\left(\mathbf X_t^\mathbf{x}\right), \mu\right) &\leq c_1 \mathrm{e}^{-c_2 t}, \quad t> 1.
\end{align}

\end{lemma}
\begin{proof}
(i)  Applying \eqref{Vbeta}, we can have the following differential inequality analogous to \eqref{Vbeta1},
\begin{equation}\label{Vbeta2}
	\frac{\mathrm{d}}{\mathrm{d} t} \mathbb{E} V_\beta\left(\mathbf X^\mathbf{x}_t\right) \leqslant-C_1 \mathbb{E}V_\beta\left(\mathbf X^\mathbf{x}_t\right)^{1+\frac{\theta}{\beta}}+C_2 \leqslant-C_1\big[\mathbb{E}V_\beta\left(\mathbf X_t^\mathbf{x}\right) \big]^{1+\frac{\theta}{\beta}}+C_2,
\end{equation}
where the last inequality uses the Jensen's inequality. For simplicity, we let $g(t) = \E  V_\beta(\mathbf X^\mathbf{x}_t)$. It follows that $g(t) \geqslant 1$, and \eqref{Vbeta2} can be expressed as  
$$
g^{\prime}(t) \leqslant-C_1 g^{1+\frac{\theta}{\beta}}(t)+C_2\leqslant-C_1 (g(t)-1)^{1+\frac{\theta}{\beta}}+C_2.
$$
By the comparison theorem, we obtain that
\begin{align}\label{barg}
g(t) -1\leqslant \bar g(t),
\end{align}
where the function $\bar g$ solves the following equation:
$$
\bar g^{\prime}(t) =-C_1 \bar g^{1+\frac{\theta}{\beta}}(t)+C_2,\quad \text{with } \ \bar g(0)=g(0)-1.
$$
Then we know $\bar g(t)  \geqslant 0$. Due to ${\theta}/{\beta}>0$, it is easy to check that $\bar g$ is monotone and 
	\begin{equation}
		\label{e:comp_lim}
		\lim_{t \to \infty} \bar g(t) = \left( {C_2}/{C_1} \right)^{1/(1+\frac{\theta}{\beta})}.
	\end{equation}

We shall divide the proof of \eqref{e:Lem2.4_1} in the following two cases:
\begin{itemize}
	\item[(a)] For the case that $0\leq \bar g(0) \leq  ( C_2 / C_1 )^{1/(1+\frac{\theta}{\beta})}$, $\bar g$ is increasing in $[0, \infty)$. Then, \eqref{barg} and \eqref{e:comp_lim}  imply that $ g(t) \leqslant \bar g(t)+1\leq  ( C_2 / C_1  )^{1/(1+\frac{\theta}{\beta})}+1$ for all $t \geq 0$.
	\item[(b)] For the case that $\bar g(0) >  ( C_2 / C_1 )^{1/(1+\frac{\theta}{\beta})}$, $\bar g$ is decreasing in $[0, \infty)$. Let
	\begin{align}\label{tau} 
		\tau := \inf \left\{ t: \bar g(t) \leqslant ( 2 C_2 / C_1 )^{1/(1+\frac{\theta}{\beta})} \right\}  . 
	\end{align}
	
	If $\tau \leq 1$, there exists a constant $t_0 \leq 1$ such that $\bar g(t_0) \leq \big( 2 C_2 / C_1 \big)^{1/(1+\frac{\theta}{\beta})}$. By the monotone of $\bar g$, it holds that 
	$$\bar g(t) \leq \left( {2 C_2}/{ C_1 } \right)^{1/(1+\frac{\theta}{\beta})}, \quad \text{for all}\ t \geq t_0.$$
	
	If $\tau > 1$, then $\bar g(t) > \big( 2 C_2 / C_1 \big)^{1/(1+\frac{\theta}{\beta})}$ for any $t\leqslant 1$, and the following differential inequality holds for all $t\leq 1$, 
	$$
	\frac{\bar g^{\prime}(t)}{\bar g^{1+\frac{\theta}{\beta}}(t)} =-C_1+\frac{C_2}{\bar g(t)^{1+\frac{\theta}{\beta}}} \leqslant-C_1+\frac{C_2}{ (( {2C_2}/{C_1})^{1/(1+\frac{\theta}{\beta})} )^{1+\frac{\theta}{\beta}}}\leq -\frac{C_1}2,
	$$
	 which leads to
	\[
		-\frac{\beta}{\theta}\left(\bar g^{-\frac{\theta}{\beta}}(t)\right)^{\prime} \leq-\frac{1}{2} C_1,\quad \text{for all}\ t\leq 1.
	\]	
	Then we know that 
	\begin{equation}\label{gttt3}
		-\frac{\beta}{\theta} \bar g^{-\frac{\theta}{\beta}}(t)+\frac{\beta}{\theta} \bar g^{-\frac{\theta}{\beta}}(0) \leq-\frac{1}{2} C_1 t,\quad \text{for all}\ t\leq 1.
	\end{equation}
	 
	By the comparison's theorem and the monotone of $\bar g$, we can obtain that for any $t\geqslant 1$,
	\begin{align*}
		g(t) \leqslant \bar g(t)+1 \leqslant \bar g(1) +1\leq \left[ \bar g^{-\frac{\theta}{\beta}}(0) +  \frac{\theta C_1}{2\beta} \right]^{-\beta/\theta}+1 \leq \left[  \frac{\theta C_1}{2\beta} \right]^{-\beta/\theta}+1 .
	\end{align*}
\end{itemize}

	Combining above all, there exists a positive constant $C$ independent of $\mathbf{x}$ such that
\begin{align*}
\sup _{t \geqslant 1} \mathbb{E}\left[V_\beta\left(\mathbf X^\mathbf{x}_t\right)\right] \leqslant C, \quad \forall \ \mathbf{x} \in \mathbb{R}^d.
\end{align*}

(ii) Next, we provide the proof of \eqref{e:Lem2.4_2}, which is also divided in two part.
\begin{itemize}
	\item[(a)] For the case that $0\leq \bar g(0) \leq ( C_2 / C_1 )^{1/(1+\frac{\theta}{\beta})}$, 
	it is easy to know that
	\begin{align}\label{gtt}
		\int_0^1  g(t) \mathrm{d} t\leq \int_0^1 \bar g(t) \mathrm{d} t+1\leq ( C_2 / C_1  )^{1/(1+\frac{\theta}{\beta})}+1.
	\end{align}
	\item[(b)] For the case that $\bar g(0) > ( C_2 / C_1 )^{1/(1+\frac{\theta}{\beta})}$, let $\tau$ be defined in \eqref{tau}.
	When $\tau \in(0,1)$ and $t \in(\tau, 1]$, since $\bar g^{\prime}(t) \leqslant C_2$, we have 
	 \begin{align}\label{gt2} 
			g(t) &\leq\bar g(t)+1 \leq C_2(1-\tau)+ ( 2 C_2 / C_1 )^{1/(1+\frac{\theta}{\beta})}+1\\
			&\leq ( 2 C_2 / C_1 )^{1/(1+\frac{\theta}{\beta})}+C_2+1 , \quad t \in(\tau, 1]. \nonumber
	\end{align} 
	When $t \in [0,\tau\wedge 1]$, $\bar g(t)>\left(2 C_2 / C_1\right)^{{1}/(1+\frac{\theta}{\beta}})$, by \eqref{gttt3}, we have
	\begin{equation}\label{gt3}  
		\begin{aligned}
			g(t) &\leqslant \bar{g}(t) + 1 \leq \left( {\bar g^{-\frac{\theta}{\beta}}(0)+\frac{1}{2} \frac{\theta}{\beta}C_1 t}\right)^{-\frac{\beta}{\theta}}+1 \\
			&\leq \left({(1+|\mathbf{x}|^2)^{-\frac{\theta}{2 }} +\frac{1}{2} \frac{\theta}{\beta} C_1 t}\right)^{-\frac{\beta}{\theta}}+1,\quad t \in [0,\tau\wedge 1].
		\end{aligned}
	\end{equation}
    Thus, for the case that $\bar g(0) > \left(C_2 / C_1\right)^{1 /\left(1+\frac{\theta}{\beta}\right)}$, by \eqref{gt2} and \eqref{gt3}, we have
\begin{align*}
\int_0^1 g(t) \mathrm{d} t
&\leq \int_0^{\tau\wedge 1} \left( {(1+|\mathbf{x}|^2)^{-\frac{\theta}{2 }} +\frac{1}{2}\frac{\theta}{\beta} C_1 t}\right)^{-\frac{\beta}{\theta}} \mathrm{d} t+\int_{\tau\wedge 1}^1 (( 2 C_2 / C_1 )^{1/(1+\frac{\theta}{\beta})}+C_2)\mathrm{d} t + 2 \\
&\leq \int_0^{\tau\wedge 1} \left( {(1+|\mathbf{x}|^2)^{-\frac{\theta}{2 }} +\frac{1}{2} \frac{\theta}{\beta} C_1 t}\right)^{-\frac{\beta}{\theta}} \mathrm{d} t+C\\
&\leq \int_0^{(1+|\mathbf{x}|^2)^{-\frac{\theta}{2}}} \left( {(1+|\mathbf{x}|^2)^{-\frac{\theta}{2 }} }\right)^{\frac{\beta}{\theta}} \mathrm{d} t+\int_{(1+|\mathbf{x}|^2)^{-\frac{\theta}{2}}}^{1} \left( {\frac{\theta C_1}{2\beta} t}\right)^{-\frac{\beta}{\theta}} \mathrm{d} t+C\\
&\leq C(1+|\mathbf{x}|^2)^{\frac{\beta- \theta }{2}}+C.
\end{align*}
\end{itemize}
Combining both of the above cases, we can get the desired result \eqref{e:Lem2.4_2} .

(iii) For $f\in \mathrm{Lip}(1)$, $\mu(f)<C\mu(V_\beta)<\infty$ immediately follows from \eqref{e:Lem2.4_1}. Lemma \ref{contract} (ii) implies that, for $t> 1$,
\begin{align*}
\mathbb W_1\left(\mathcal{L}\left(\mathbf X_t^\mathbf{x}\right), \mu\right)&=\sup _{f\in \mathrm{Lip}(1)}\left|\mathrm{P}_tf(\mathbf{x})-\mu(f)\right|\\
&=\sup _{f\in \mathrm{Lip}(1)}\left|\mathrm{P}_tf(\mathbf{x})-\int_{\mathbb R^d} \mathrm{P}_tf(\mathbf y)\mu(\mathrm d \mathbf y)\right|\\
&\leq \sup _{f\in \mathrm{Lip}(1)}\int_{\mathbb R^d}\left|\mathrm{P}_tf(\mathbf{x})- \mathrm{P}_tf(\mathbf y)\right|\mu(\mathrm d \mathbf y)\\
&=\int_{\mathbb R^d}\mathbb W_1\left(\mathcal{L}\left(\mathbf X_t^\mathbf{x}\right), \mathcal{L}\left(\mathbf X_t^\mathbf y\right)\right)\mu(\mathrm d \mathbf y) \leq c_1 \mathrm{e}^{-c_2 t}.
\end{align*}

The proof is complete.
\end{proof}

\section{The proof of Main Results}{\label{ProofResults}}

\subsection{The proof of Theorem \ref{thm:Limit1}}{\label{ProofResults1}}

 Consider the Poisson equation \eqref{Poisson} with a bounded measurable test function, that is: for $h \in \mathcal{B}_b(\mathbb{R}^d, \mathbb{R})$,
\begin{equation}
	\label{e:b_Poi}
	\mathcal{A} f_{h}({\bf x}) = h(\bf x) - \mu(h),
\end{equation}
where $\mathcal{A}$ is defined by \eqref{generator} and $\mu$ is the unique ergodic measure for SDE \eqref{e:GSDE}. 
We first analyze the regularity of the solution to \eqref{e:b_Poi}, as stated in the following lemma.  

\begin{lemma}\label{lem:reghf}
	Consider \eqref{e:b_Poi}. For function $h\in \mathcal{B}_b(\R^d,\R)$, one has,
	
	(i) A solution to \eqref{e:b_Poi} is given by
	\begin{eqnarray} {\label{Lip2}}
		f_h(\mathbf{x})= - \int_0^{\infty} \mathrm{P}_t[   h(\mathbf{x}) - \mu(h) ] \dif t,
	\end{eqnarray}
	which implies that for some $C>0$, and $\beta<\alpha/4$,
	\begin{align}\label{e:reghf0}
	    \left|f_h(\mathbf{x})\right| &\leqslant C\left(1+|\mathbf{x}|^{ 2}\right)^{ \beta/2}. 
	\end{align}
     
	(ii) The solution to \eqref{e:b_Poi} is also given by
	\begin{equation} \label{e:steinequs3}
		f_h(\mathbf{x})=  \int_0^{\infty} {\rm e}^{-a t} \mathrm{P}_t[  a f_h(\mathbf{x}) - h(\mathbf{x}) + \mu(h) ] \dif t,  \quad  \forall  a>0,
	\end{equation}
	which implies that  for some $C>0$, and $\beta<\alpha/4$,
	\begin{align}\label{e:reghNf0}
    \left|\nabla f_h(\mathbf{x})\right| &\leqslant C\left(1+|\mathbf{x}|^2\right)^{ \beta/2}.
	\end{align}
\end{lemma}
\begin{proof} 
	(i) We first show that $\int_0^{\infty} [\mathrm{P}_t h(\mathbf{x})- \mu(h)] \dif t$ is well defined. Due to $\|h\|_{\infty}\leq C$, it follows from Lemma \ref{LemmaErgodicity0} that
\begin{align}{\label{fh1}}
\Big| \int_0^{\infty} \big[\mathrm{P}_t h(\mathbf{x}) - \mu(h) \big] \dif t \Big|& \leqslant \int_0^{\infty}\left|\mathrm{P}_t h(\mathbf{x})-\mu(h)\right| \mathrm{d} t \nonumber  \\
&\leqslant \|h\|_{\infty}\int_0^{\infty}c_1 V_\beta(\mathbf{x}) \mathrm{e}^{-c_2 t}\mathrm{d} t \\
& \leqslant C  V_\beta (\mathbf{x})  \leqslant C \left(1+|\mathbf{x}|^2\right)^{\beta/2} <\infty.\nonumber
\end{align}
Let $R_a=(a-\mathcal{A})^{-1}$ be the resolvent operator associate with $\mathcal{A}$ (see, \cite{applebaum2009levy}). Then, for all $a>0$, we have
\begin{align*}
R_a (h(\mathbf{x})-\mu(h)))=\int_0^{\infty} {\rm e}^{-a t} \mathrm{P}_t (h(\mathbf{x})-\mu(h)) \mathrm{d} t,
\end{align*}
which leads to
\begin{align*}
a \int_0^{\infty} {\rm e}^{-a t} \mathrm{P}_t (h(\mathbf{x})-\mu(h)) \mathrm{d} t-(h(\mathbf{x})-\mu(h))=\mathcal{A}\left(\int_0^{\infty} {\rm e}^{-a t} \mathrm{P}_t (h(\mathbf{x})-\mu(h)) \mathrm{d} t\right).
\end{align*}
As $a \rightarrow 0+$, by \eqref{fh1} and dominated convergence theorem, one has
\begin{align*}
&a \int_0^{\infty} {\rm e}^{-a t} \mathrm{P}_t (h(\mathbf{x})-\mu(h)) \mathrm{d} t-(h(\mathbf{x})-\mu(h)) \rightarrow-(h(\mathbf{x})-\mu(h)),\\
&\int_0^{\infty} {\rm e}^{-a t} \mathrm{P}_t (h(\mathbf{x})-\mu(h)) \mathrm{d} t \rightarrow \int_0^{\infty} \mathrm{P}_t (h(\mathbf{x})-\mu(h)) \mathrm{d} t.
\end{align*}
By the fact that $\mathcal{A}$ is a closed operator, we know that $\int_0^{\infty} \mathrm{P}_t (h(\mathbf{x})-\mu(h)) \mathrm{d} t$ is in the domain of $\mathcal{A}$ and
\begin{align*}
-(h(\mathbf{x})-\mu(h))=\mathcal{A}\left(\int_0^{\infty} \mathrm{P}_t (h(\mathbf{x})-\mu(h)) \mathrm{d} t\right).
\end{align*}
Thus, we obtain \eqref{Lip2}, that is,
\begin{align*}
f_h(\mathbf{x})=-\int_0^{\infty}  \mathrm{P}_t[h(\mathbf{x})-\mu(h)] \mathrm{d} t .
\end{align*}
Additionally, \eqref{fh1} yields \eqref{e:reghf0} immediately.
	
(ii) Since $ (a - \mathcal{A} ) f_h(\mathbf{x}) = a f_h(\mathbf{x}) - h(\mathbf{x}) + \mu(h) $, by $R_a=(a-\mathcal{A})^{-1}$, one has
\begin{equation*}
f_h(\mathbf{x}) = (a  - \mathcal{A} )^{-1} [ a f_h(\mathbf{x}) - h(\mathbf{x}) + \mu(h)]  = R_{a} [ a f_h(\mathbf{x}) - h(\mathbf{x}) + \mu(h)],
\end{equation*}
which implies \eqref{e:steinequs3} holds true, that is,
\begin{equation*}
f_h(\mathbf{x})=  \int_0^{\infty} {\rm e}^{-a t} \mathrm{P}_t[  a f_h(\mathbf{x}) - h(\mathbf{x}) + \mu(h) ] \dif t,  \quad  \forall  a>0.
\end{equation*}
    
We divide the proof of \eqref{e:reghNf0} into two cases: $\theta>0$ and $\theta = 0$. 
   
\underline{In the case of $\theta>0$}, since $h \in \mathcal{B}_b(\R^d,\R)$, by Lemma \ref{LemmaErgodicity1}, one has that for all $t>0$
\begin{equation*} 
|\mathrm{P}_t h(\mathbf{x}) - \mu(h) |
\leq \| h \|_{\infty} \| \mathcal{L}(\mathbf X^\mathbf{x}_t) - \mu  \|_{ \rm{TV} }
\leq C\| h \|_{\infty} {\rm e}^{-c t} ,
\end{equation*}
which implies that
\begin{equation*}
|f_h(\mathbf{x})|\leq \int_0^{\infty} \big| \mathrm{P}_t[ h(\mathbf{x}) - \mu(h) ] \big| \dif t\le  \int_{0}^{\infty} C  {\rm e}^{-c t} \|h\|_{\infty} \dif t \le \ C \|h\|_{\infty}.
\end{equation*}
That is, $f_h$ and $h$ are both in $\mathcal{B}_b(\R^d, \R)$. By \cite[Corollary 2.2]{wang2014harnack}, one has 
	\begin{equation*}
		|\nabla \mathrm{P}_t f_h(\mathbf{x})| \leq \frac{C \| f _h\|_{\infty} }{ (1\wedge t)^{{1}/{\alpha}}} < \infty 
		\quad \text{ and} \quad
		| \nabla \mathrm{P}_t h(\mathbf{x})|  \leq  \frac{C\| h \|_{\infty}}{ (1\wedge t)^{{1}/{\alpha}} } <\infty.
	\end{equation*}
	Combining these with \eqref{e:steinequs3} and using the dominated convergence theorem, one has
	\begin{equation*}
		\nabla f_h(\mathbf{x}) =  \int_0^{\infty} {\rm e}^{-a t} \big[ a \nabla \mathrm{P}_t f_h(\mathbf{x}) - \nabla \mathrm{P}_t h(\mathbf{x}) \big] \dif t,
	\end{equation*}
	which implies that for $\theta>0$, for all $\mathbf{x}\in \R^d$,
	\begin{equation}\label{e:b_Gfh1}
		\begin{split}
			|\nabla f_h(\mathbf{x})|
			&\leq \int_0^{\infty} {\rm e}^{-a t} \big[ a |\nabla \mathrm{P}_t f_h(\mathbf{x})| + | \nabla \mathrm{P}_t h(\mathbf{x})| \big] \dif t  \\
			&\leq \int_0^{\infty} {\rm e}^{-a t} \Big[ a \frac{C \| f_h \|_{\infty} }{ (1\wedge t)^{{1}/{\alpha}}} + \frac{C\| h \|_{\infty}}{ (1\wedge t)^{{1}/{\alpha}} } \Big] \dif t  
			\leq C \| h \|_{\infty}.
		\end{split}
	\end{equation}
    
\underline{In the case of $\theta=0$}, for any $h\in \mathcal{B}_b(\R^d,\R)$, we define 
\begin{align}{\label{hdelta}}
h_\delta(\mathbf{x})=\int_{\mathbb{R}^d} \varphi_\delta(\mathbf{y}) h(\mathbf{x}-\mathbf{y}) \mathrm{d} \mathbf{y}, \quad \delta>0,
\end{align}
where $\varphi_\delta$ is the density function of the normal distribution $\mathcal{N}\left(\mathbf{0}, \delta^2 \mathbf{I}_d\right)$. 
 
 Thus $h_\delta$  is smooth, $\lim _{\delta \rightarrow 0} h_\delta(\mathbf{x})=h(\mathbf{x})$, and   $\| h_\delta \|_{\infty}  \leqslant \|h\|_{\infty}$ . It follows from \eqref{Lip2} that the solution to \eqref{e:b_Poi} with respect to $h_\delta$, denoted by $f_{h, \delta}$, can be expressed as 
\begin{align}{\label{stein1}}
f_{h, \delta}(\mathbf{x})=-\int_0^{\infty} \mathrm{P}_t\left[h_\delta(\mathbf{x})-\mu\left(h_\delta\right)\right] \mathrm{d} t .
\end{align}
 Moreover, by \eqref{e:steinequs3}, the solution $f_{h, \delta}$ can also be expressed as 
\begin{align}{\label{stein2}}
f_{h, \delta}(\mathbf{x})=\int_0^{\infty} \mathrm{e}^{-a t} \mathrm{P}_t\left[a f_{h, \delta}(\mathbf{x})-h_{ \delta}(\mathbf{x})+\mu(h_{ \delta})\right] \mathrm{d} t, \quad \forall a>0 .
\end{align}

 For all $\mathbf{u} \in \mathbb{R}^d$, \eqref{stein1} leads to 
\begin{align}\label{gradfh}
\nabla_{\mathbf{u}} f_{h, \delta}(\mathbf{x})=\int_0^{\infty} \nabla_{\mathbf{u}} \mathbb{E} h_\delta\left(\mathbf{X}_t^{\mathbf{x}}\right) \mathrm{d} t=\int_0^{\infty} \mathbb{E}\left[\nabla h_\delta\left(\mathbf{X}_t^{\mathbf{x}}\right) \nabla_{\mathbf{u}} \mathbf{X}_t^{\mathbf{x}}\right] \mathrm{d} t ,
\end{align}
where $\nabla_{\mathbf{u}}\mathbf X_t^\mathbf{x}:=\lim _{\epsilon \rightarrow 0} (\mathbf X_t^{\mathbf{x}+\epsilon \mathbf{u}}-\mathbf X_t^\mathbf{x})/{\epsilon}, t \geq 0$.  This limit exists and satisfies 
$$
\frac{\mathrm{d}}{\mathrm{d} t} \nabla_{\mathbf{u}} \mathbf X_t^\mathbf{x}= \nabla_{\nabla_{\mathbf{u}} \mathbf X_t^\mathbf{x}}\mathbf  b\left(\mathbf X_t^\mathbf{x}\right), 
 \quad \text{and} \quad
\nabla_{\mathbf{u}} \mathbf X_0^\mathbf{x}=\mathbf{u}. 
$$
By Assumption \ref{a:Assumption}, it holds that
\begin{align*}
\frac{\mathrm{d}}{\mathrm{d} t}\left|\nabla_{\mathbf{u}} \mathbf X_t^\mathbf{x}\right|^2&=2\left\langle\nabla_{\mathbf{u}} \mathbf X_t^\mathbf{x}, \nabla_{\nabla_{\mathbf{u}} \mathbf X_t^\mathbf{x}} \mathbf b\left(\mathbf X_t^\mathbf{x}\right)\right\rangle\\
&= 2\left\langle\lim _{\epsilon \rightarrow 0} (\mathbf X_t^{\mathbf{x}+\epsilon \mathbf{u}}-\mathbf X_t^\mathbf{x})/{\epsilon}, \lim _{\epsilon \rightarrow 0} (\mathbf b(\mathbf X_t^{\mathbf{x}+\epsilon \mathbf{u}})-\mathbf b(\mathbf X_t^\mathbf{x}))/{\epsilon}\right\rangle\\
&=2\lim _{\epsilon \rightarrow 0}\frac{1}{\epsilon^2}\left\langle \mathbf X_t^{\mathbf{x}+\epsilon \mathbf{u}}-\mathbf X_t^\mathbf{x} ,  \mathbf b(\mathbf X_t^{\mathbf{x}+\epsilon \mathbf{u}})-\mathbf b(\mathbf X_t^\mathbf{x})  \right\rangle\\
&\leq 2 K_1\left|\nabla_{\mathbf u} \mathbf X_t^\mathbf{x}\right|^2.
\end{align*}
 Applying Gr\"onwall's inequality, we know $\left|\nabla_{\mathbf{u}}\mathbf X_t^\mathbf{x}\right|^2 \leq \mathrm{e}^{2K_1t}\left|\mathbf u\right|^2$ for all ${\bf u} \in \mathbb{R}^d$.  Hence, ${f}_{h,\delta}$ is in $\mathcal{C}^1\left(\mathbb{R}^d, \mathbb{R}\right)$ by \eqref{gradfh}. Then it follows from \cite[ Theorem 1.1]{zhang2013derivative} and \eqref{stein2} that for any $\mathbf{u} \in \mathbb{R}^d$,
\begin{align}{\label{stein3}}
\nabla_{\mathbf{u}} {f}_{h,\delta}(\mathbf{x})=\int_0^{\infty} \mathrm{e}^{-a t} \mathbb{E}\left[\frac{1}{S_t}\left(a {f}_{h,\delta}\left(\mathbf{X}_t^{ \mathbf{x}}\right)-h_\delta\left(\mathbf{X}_t^{ \mathbf{x}}\right)+\mu (h_\delta)\right)\right] \int_0^t  \nabla_{\mathbf{u}} \mathbf{X}_s^{ \mathbf{x}} \mathrm{d} W_{S_s} \mathrm{ d} t,
\end{align}
where the subordinated $\left(W_{S_t}\right)_{t \geqslant 0}$ is the rotationally symmetric $\alpha$ stable process.

By Lemma \ref{LemmaErgodicity0} , with similar calculations for \eqref{fh1}, we have
\begin{align}{\label{fhdelta}}
\left|f_{h,\delta}(\mathbf{x})\right| & \leqslant \int_0^{\infty}\left|\mathrm{P}_t h_{\delta}(\mathbf{x})-\mu(h_{\delta})\right| \mathrm{d} t \leqslant \|h\|_{\infty}\int_0^{\infty}c_1 V_\beta(\mathbf{x}) \mathrm{e}^{-c_2 t}\mathrm{d} t \\
& \leqslant C  V_\beta (\mathbf{x})  \leqslant C \left(1+|\mathbf{x}|^2\right)^{\beta/2}. \nonumber
\end{align}
 Then, by the moment estimate for $\mathbb{E}\left|\mathbf{X}_s^{\mathbf{x}}\right|$ in Lemma \ref{LemmaErgodicity0} and the fact $2 \beta<\alpha$, we have the moment estimate for $\mathbb{E}\left|{f_{h,\delta}}\left(\mathbf{X}_s^{\mathbf{x}}\right)\right|^2$, that is,
$$
\mathbb{E}\left|f_{h,\delta}\left(\mathbf{X}_s^{\mathbf{x}}\right)\right|^2 \leqslant C \mathbb{E}\left(1+\left|\mathbf{X}_s^{\mathbf{x}}\right|^2\right)^{\beta } \leqslant   C\left(1+|\mathbf{x}|^{2}\right)^\beta. 
$$
 Using \cite[Theorem 1.1]{zhang2013derivative} and $|h_\delta(\mathbf{x})|  \leqslant \|h\|_{\infty}$, there exists some positive constant $C$ such that, for any $|\mathbf{u}|\leq 1$, and $a> K_1+1$,
\begin{align}{\label{grafhdelta}}
&\left|\nabla_{\mathbf{u}} {f}_{h,\delta}(\mathbf{x})\right| \nonumber\\
= &\int_0^{\infty} \left|\mathrm{e}^{-a t} \mathbb{E}\left[\frac{1}{S_t}\left(a {f}_{h,\delta}\left(\mathbf{X}_t^{ \mathbf{x}}\right)-h_{\delta}\left(\mathbf{X}_t^{ \mathbf{x}}\right)+\mu (h_{\delta})\right)\right] \int_0^t  \nabla_{\mathbf{u}} \mathbf{X}_s^{ \mathbf{x}} \mathrm{d} W_{S_s} \right|\mathrm{ d}t\\
\leq & \int_0^{\infty} \mathrm{e}^{-a t}\left[a\left(\mathbb{E}\left|{f_{h,\delta}}\left(\mathbf{X}_t^{\mathbf{x}}\right)\right|^2\right)^{1 / 2}+\left(\mathbb{E}\left|h_{\delta}\left(\mathbf{X}_t^{\mathbf{x}}\right)-\mu(h_{\delta})\right|^2\right)^{1 / 2}\right] \mathrm{e}^{ K_1t} t^{-1 / \alpha} \mathrm{d} t \nonumber\\
\leqslant & C\left(1+|\mathbf{x}|^{2}\right)^ {\beta/2}.\nonumber
\end{align}

In the final step, we prove that $\lim _{\delta \rightarrow 0} \nabla_{\mathbf{u}} f_{h, \delta}(\mathbf{x})=\nabla_{\mathbf{u}} f_h(\mathbf{x})$, which implies $\left|\nabla_{\mathbf{u}} f_h(\mathbf{x})\right| \leqslant C\left(1+|\mathbf{x}|^{2}\right)^ {\beta/2} $ by \eqref{grafhdelta}, and the proof will be complete. 

By \eqref{fhdelta} and the dominated convergence theorem, we have
\begin{align}{\label{limdelta}}
\lim _{\delta \rightarrow 0} f_{h, \delta}(\mathbf{x})=-\int_0^{\infty} \mathrm{P}_t[h(\mathbf{x})-\mu(h)] \mathrm{d} t=f_h(\mathbf{x}) .
\end{align}
It follows from \eqref{stein3}, \eqref{grafhdelta}, \eqref{limdelta} and the dominated convergence theorem that
\begin{align*}
& \lim _{\delta \rightarrow 0} \nabla_{\mathbf{u}} f_{h, \delta}(\mathbf{x}) \\
= & \lim _{\delta \rightarrow 0} \int_0^{\infty} \mathrm{e}^{-a t} \mathbb{E}\left[\frac{1}{S_t}\left(a f_{h, \delta}\left(\mathbf{X}_t^{\mathbf{x}}\right)-h_\delta\left(\mathbf{X}_t^{\mathbf{x}}\right)+\mu\left(h_\delta\right)\right)\right] \int_0^t  \nabla_{\mathbf{u}} \mathbf{X}_s^{\mathbf{x}} \mathrm{d} W_{S_s} \mathrm{~d} t \\
= & \int_0^{\infty} \mathrm{e}^{-a t} \mathbb{E}\left[\frac{1}{S_t}\left(a f_h\left(\mathbf{X}_t^{\mathbf{x}}\right)-h\left(\mathbf{X}_t^{\mathbf{x}}\right)+\mu\left(h\right)\right)\right] \int_0^t \nabla_{\mathbf{u}} \mathbf{X}_s^{\mathbf{x}} \mathrm{d} W_{S_s} \mathrm{d} t .
\end{align*}
 Then, combining \eqref{limdelta} and above equation, by the fact that the operator $\nabla$ is closed, we know that $f_h(\mathbf{x})$ is in the domain of $\nabla$ and
$$
\lim _{\delta \rightarrow 0} \nabla_{\mathbf{u}} f_{h, \delta}(\mathbf{x})=\nabla_{\mathbf{u}} f_h(\mathbf{x}).
$$
Thus, \eqref{e:reghNf0} holds by \eqref{grafhdelta} immediately.
\end{proof}

\begin{proof}[Proof of Theorem \ref{thm:Limit1}. ]
By Itô's formula for function $f_h$ which is the solution to Poisson equation \eqref{e:b_Poi} , one has
\begin{align*} 
f_h\left(\mathbf X_t^\mathbf{x}\right)-f_h(\mathbf{x}) 
=\int_0^t \mathcal{A} f_h\left(\mathbf X_s^\mathbf{x}\right) \mathrm{d} s+\int_0^t \int_{\mathbb{R}^d_0}\left[f_h\left(\mathbf X_{s-}^\mathbf{x}+\mathbf z\right)-f_h\left(\mathbf X_{s-}^\mathbf{x}\right)\right] \widetilde{N}(\mathrm{d} \mathbf z , \mathrm{d} s).\nonumber
\end{align*}
Combining above equation with Poisson equation \eqref{e:b_Poi}, one has
\begin{align}\label{I1I2}
& \sqrt{t}\left[ \mathcal{E}_t^{\mathbf{x}}(h)  -\mu(h)\right]=\frac{1}{\sqrt{t}} \int_0^t \mathcal{A} f_h\left(\mathbf X_s^\mathbf{x}\right) \mathrm{d}s\\
= & \frac{1}{\sqrt{t}}\left[f_h\left(\mathbf X_t^\mathbf{x}\right)-f_h(\mathbf{x})\right]-\frac{1}{\sqrt{t}} \int_0^t \int_{\mathbb{R}^d_0}\left[f_h\left(\mathbf X_{s-}^\mathbf{x}+\mathbf z\right)-f_h\left(\mathbf X_{s-}^\mathbf{x}\right)\right] \widetilde{N}(\mathrm{d} \mathbf z, \mathrm{d} s)\nonumber\\
=&:I_1+I_2 .\nonumber
\end{align}

For $I_1$, we know $\left|f_h(\mathbf{x})\right| \leqslant C\left(1+|\mathbf{x}|^{ 2}\right)^{ \beta/2}$ by Lemma \ref{lem:reghf} (i). Combining Lemma \ref{LemmaErgodicity0}, we can obtain that 
\begin{align}\label{I01}
\mathbb E\left|\frac{1}{\sqrt{t}}\left[f_h\left(\mathbf X_t^\mathbf{x}\right)-f_h(\mathbf{x})\right]\right| \leq \frac{1}{\sqrt{t}}\mathbb E\left[\left| f_h\left(\mathbf X_t^\mathbf{x}\right)\right|+\left|f_h(\mathbf{x}) \right|\right] \rightarrow 0, \quad \text { as } t \rightarrow \infty.
\end{align}

For $I_2$, It follows from Lemma \ref{lem:reghf} that for some $C>0$ and $\beta<\alpha/4$,
\begin{align*}
& \int_{\mathbb{R}_0^d}\left[f_h(\mathbf{x}+ \mathbf{z})-f_h(\mathbf{x})\right]^2 \nu(\mathrm{d} \mathbf{z}) \\
= & \int_{|\mathbf{z}| \leqslant 1}\left|\int_0^1\left\langle\nabla f_h(\mathbf{x}+r  \mathbf{z}),  \mathbf{z}\right\rangle \mathrm{d} r\right|^2 \nu (\mathrm{d}\mathbf {z})+\int_{|\mathbf{z}|>1}\left[f_h(\mathbf{x}+ \mathbf{z})-f_h(\mathbf{x})\right]^2 \nu(\mathrm{d} \mathbf{z}) \\
\leqslant & C \int_{|\mathbf{z}| \leqslant 1}\left(1+|\mathbf{x}|^{\beta}+|\mathbf{z}|^{\beta}\right)^2|\mathbf{z}|^2 \nu(\mathrm{d} \mathbf{z}) + C\int_{|\mathbf{z}|>1}\left(1+|\mathbf{x}|^{\beta}+|\mathbf{z}|^{\beta}\right)^2 \nu(\mathrm{d} \mathbf{z}),
\end{align*}
which implies that
\begin{align}{\label{VV}}
& \mathcal V\left(f_h\right):=\int_{\mathbb{R}^d} \int_{\mathbb{R}^d_0}\left[f_h(\mathbf{x}+ \mathbf{z})-f_h(\mathbf{x})\right]^2 \nu(\mathrm{d}\mathbf z) \mu(\mathrm{d}\mathbf{x}) \\
 \leqslant& C \int_{\mathbb{R}^{d}}\left[\int_{|\mathbf z|\leq 1}\left(1+|\mathbf{x}|^{\beta}+|\mathbf z|^{\beta}\right)^2|\mathbf z|^2 \nu(\mathrm{d}\mathbf z)+\int_{|\mathbf z|>1}\left(1+|\mathbf{x}|^{\beta}+|\mathbf z|^{\beta}\right)^2 \nu(\mathrm{d}\mathbf z)\right] \mu(\mathrm{d}\mathbf{x})\nonumber \\
\leqslant&  C  \left[ 1 + \mu\left(|\mathbf{x}|^{2\beta}\right)\right]  <\infty\nonumber
\end{align}
where the last two inequalities hold from the fact $2\beta<\alpha$ and the fact that $\nu(\mathrm{d} \mathbf{z})=C_{d, \alpha}|\mathbf{z}|^{-\alpha-d} \mathrm{d} \mathbf{z}$.

To make notations simple, we denote
\begin{equation}\label{e:U}
	\begin{aligned}
		U_i & =\int_{i-1}^i 	\int_{\mathbb{R}_0^d}\left[f_h\left(\mathbf{X}_{s-}^{\mathbf{x}}+ \mathbf{z}\right)-f_h\left(\mathbf{X}_{s-}^{\mathbf{x}}\right)\right] \widetilde{N}(\mathrm{d} \mathbf{z},\mathrm{d} s) \quad \text { for } i=1,2, \cdots,\lfloor t\rfloor, \\
		U_{\lfloor t\rfloor+1} & =\int_{\lfloor t\rfloor}^t 	\int_{\mathbb{R}_0^d}\left[f_h\left(\mathbf{X}_{s-}^{\mathbf{x}}+ \mathbf{z}\right)-f_h\left(\mathbf{X}_{s-}^{\mathbf{x}}\right)\right] \widetilde{N}(\mathrm{d} \mathbf{z},\mathrm{d} s).
\end{aligned}
\end{equation}
Then, we have
\begin{align*}
I_2 = -\frac{1}{\sqrt{t}} \int_0^t \int_{\mathbb{R}_0^d}\left[f_h\left(\mathbf{X}_{s-}^{\mathbf{x}}+\mathbf{z}\right)-f_h\left(\mathbf{X}_{s-}^{\mathbf{x}}\right)\right] \widetilde{N}(\mathrm{d} \mathbf{z},\mathrm{d} s )=-\sum_{i=1}^{\lfloor t\rfloor+1} U_i.
\end{align*}

We know that $\{U_i\}$ are martingale differences. By Lemma \ref{LemmaErgodicity0}, Lemma \ref{lem:reghf} and the calculation which is similar with inequality \eqref{VV} for all $i=1,2, \cdots,\lfloor t\rfloor+1$, we have 
 \begin{align}\label{U2}
    \mathbb{E} |U_i|^2&=\int_{i-1}^i \int_{\mathbb{R}_0^d}\mathbb{E} \left[f_h\left(\mathbf{X}_{s }^{\mathbf{x}}+\mathbf{z}\right)-f_h\left(\mathbf{X}_{s }^{\mathbf{x}}\right)\right]^2 \nu(\mathrm{d} \mathbf z)\mathrm{d} s\\
    &\leqslant C \int_{i-1}^i\left(1+\mathbb{E} \left|\mathbf{X}_{s}^{\mathbf{x}}\right|^{2\beta}\right) \mathrm{d} s \leqslant C\left(1+|\mathbf{x}|^{2}\right)^\beta<\infty.\nonumber
\end{align} 
We claim that
\begin{align}
&\lim _{t \rightarrow \infty} \mathbb{E}\left[\max _{1 \leqslant i \leqslant\lfloor t\rfloor+1} \frac{1}{\sqrt{\mathcal{V}\left(f_h\right) t}}\left|U_i\right|\right]=0,\label{UU5} \\
&\sum_{i=1}^{\lfloor t\rfloor+1} \frac{1}{\mathcal{V}\left(f_h\right)t}\left|U_i\right|^2 \xrightarrow{\mathrm{P}} 1, \quad \text { as } t \rightarrow \infty .\label{UU6}
\end{align}

Applying the martingale CLT (see, \cite[Theorem 2]{sethuraman2002martingale}), one has
\begin{align}\label{martingale}
\frac{1}{\sqrt{\mathcal{V}\left(f_h\right) t}} \sum_{i=1}^{\lfloor t\rfloor+1} U_i \Rightarrow \mathcal{N}(0,1),\quad \text { as } t \rightarrow \infty.
\end{align}
By using the converging together lemma (see, \cite{durrett2019probability}), and combining \eqref{I1I2}, \eqref{I01} and \eqref{martingale}, one has the desired result, that is,
$$
\sqrt{t}\left[\mathcal{E}_t^x(h)-\mu(h)\right] \Rightarrow \mathcal{N}\left(0,\mathcal V\left(f_h\right)\right), \quad \text { as } t \rightarrow \infty.
$$

It remains to show \eqref{UU5} and \eqref{UU6}. For \eqref{UU5}, one has
\begin{align*}
\mathbb{E}\left[\max _{1 \leqslant i \leqslant\lfloor t\rfloor+1}\left|U_i\right|^2\right]& =\mathbb{E}\left[\max _{1 \leqslant i \leqslant\lfloor t\rfloor+1}\left(\left|U_i\right|^2 1_{\left\{\left|U_i\right|^2 \leqslant \sqrt t\right\}}+\left|U_i\right|^2 1_{\left\{\left|U_i\right|^2>\sqrt t\right\}}\right)\right]\\
& \leqslant \mathbb{E}\left[\max _{1 \leqslant i \leqslant\lfloor t\rfloor+1}\left|U_i\right|^2 1_{\left\{\left|U_i\right|^2 \leq \sqrt t\right\}}\right]+\mathbb{E}\left[\max _{1 \leqslant i \leqslant[t]+1}\left|U_i\right|^2 1_{\left\{\left|U_i\right|^2>\sqrt t\right\}}\right] \\
& \leqslant \sqrt t+\mathbb{E}\left[\max _{1 \leqslant i \leqslant\lfloor t\rfloor+1}\left|U_i\right|^2 1_{\left\{\left|U_i\right|^2>\sqrt t\right\}}\right]\\
& \leqslant \sqrt t+\sum_{i=1}^{\lfloor t\rfloor+1} \mathbb{E}\left[\left|U_i\right|^2 1_{\left\{\left|U_i\right|^2>\sqrt t\right\}}\right]\\
&\leqslant \sqrt t+(\lfloor t\rfloor+1) \max _{1 \leqslant i \leqslant\lfloor t\rfloor+1} \mathbb{E}\left[\left|U_i\right|^2 1_{\left\{\left|U_i\right|^2>\sqrt t\right\}}\right].
\end{align*}
By \eqref{U2}, it is easy to know that,
$$
\max _{1 \leqslant i \leqslant \lfloor t \rfloor+1} \mathbb{E}\left[\left|U_i\right|^2 1_{\left\{\left|U_i\right|^2>\sqrt{t}\right\}}\right] \rightarrow 0, \quad \text { as } t \rightarrow \infty.
$$
Thus, we can obtain \eqref{UU5} by the fact that
\begin{align*}
\mathbb{E}\left[\max _{1 \leqslant i \leqslant \lfloor t\rfloor+1} \frac{1}{\mathcal V\left(f_h\right) t}\left|U_i\right|^2\right] & \leqslant \frac{1}{\mathcal V\left(f_h\right) t}\left[\sqrt{t}+(\lfloor t\rfloor+1) \max _{1 \leqslant i \leqslant\lfloor t\rfloor+1} \mathbb{E}\left(\left|U_i\right|^2 1_{\left\{\left|U_i\right|^2>\sqrt{t}\right\}}\right)\right] \\
& \rightarrow 0, \quad \text { as } t \rightarrow \infty.
\end{align*}

It is easy to verify \eqref{UU6} by proving the following equation,
\begin{align}\label{condit2}
\lim _{t \rightarrow \infty} \mathbb{E}\left|\sum_{i=1}^{\lfloor t\rfloor+1} \frac{1}{\mathcal{V}\left(f_h\right) t}| U_i|^2-\frac{\lfloor t\rfloor+1}{t}\right|^2=0.
\end{align}
Observe that
\begin{align*}
&  \mathbb{E}\left|\sum_{i=1}^{\lfloor t\rfloor+1} \frac{1}{\mathcal{V}\left(f_h\right) t} | U_i |^2- \frac{\lfloor t\rfloor+1}{t}\right|^2=\mathbb{E}\left|\frac{1}{t} \sum_{i=1}^{\lfloor t\rfloor+1}\left(\frac{1}{\mathcal{V}\left(f_h\right)}\left|U_i\right|^2-1\right)\right|^2 \\
= & \frac{1}{t^2} \sum_{i=1}^{\lfloor t\rfloor+1} \mathbb{E}\left(\frac{1}{\mathcal{V}\left(f_h\right)}\left|U_i\right|^2-1\right)^2+\frac{2}{t^2} \sum_{i<j} \mathbb{E}\left[\left(\frac{1}{\mathcal{V}\left(f_h\right)}\left|U_i\right|^2-1\right)\left(\frac{1}{\mathcal{V}\left(f_h\right)}\left|U_j\right|^2-1\right)\right]\\
=&:J_1+J_2.
\end{align*}
By Kunita's inequality, Lemma \ref{LemmaErgodicity0}, Lemma \ref{lem:reghf} and the fact that $4\beta<\alpha$, there exists some positive constant $C$ such that
\begin{align*}
\mathbb{E}\left|U_i\right|^4 \leqslant & \mathbb{E}\left|\int_{i-1}^i \int_{\mathbb{R}_0^d}\left[f_h\left(\mathbf{X}_{s }^{\mathbf{x}}+ \mathbf{z}\right)-f_h\left(\mathbf{X}_{s}^{\mathbf{x}}\right)\right]^2 \nu(\mathrm{d} \mathbf{z}) \mathrm{d} s\right|^2 \\
& +\mathbb{E} \int_{i-1}^i \int_{\mathbb{R}_0^d}\left[f_h\left(\mathbf{X}_{s}^{\mathbf{x}}+ \mathbf{z}\right)-f_h\left(\mathbf{X}_{s}^{\mathbf{x}}\right)\right]^4 \nu(\mathrm{d} \mathbf{z}) \mathrm{d} s \\
 \leqslant& C \int_{i-1}^i\left(1+\mathbb{E}\left|\mathbf{X}_{s}^{\mathbf{x}}\right|^{4\beta}\right) \mathrm{d} s\leqslant C\left(1+|\mathbf{x}|^{4\beta}\right).
\end{align*}

As for $J_1$, combining above inequality with \eqref{VV} and \eqref{U2} , we know
\begin{align*}
\mathbb{E}\left(\frac{1}{\mathcal{V}\left(f_h\right)}\left|U_i\right|^2-1\right)^2
&= \mathbb{E}\left[\frac{1}{\mathcal{V}\left(f_h\right)^2}\left|U_i\right|^4-\frac{2}{\mathcal{V}\left(f_h\right)}\left|U_i\right|^2+1\right] \\
&\leqslant  C\left(1+|\mathbf{x}|^{4\beta}\right),
\end{align*}
which implies that
\begin{equation}
	\label{e:J1}	
J_1=\frac{1}{t^2} \sum_{i=1}^{\lfloor t\rfloor+1} \mathbb{E}\left(\frac{1}{\mathcal{V}\left(f_h\right)}\left|U_i\right|^2-1\right)^2 \rightarrow 0, \quad \text { as } t \rightarrow \infty.
\end{equation}

As for $J_2$, by Lemma \ref{LemmaErgodicity0} and Lemma \ref{lem:reghf}, for each $i$, it holds that for some constant $C>0$
\begin{align*}
& \sum_{j=i+1}^{\lfloor t\rfloor+1} \mathbb{E}\left[\left.\left(\frac{1}{\mathcal{V}\left(f_h\right)}\left|U_j\right|^2-1\right) \right\rvert\, \mathcal{F}_i\right]=\frac{1}{\mathcal{V}\left(f_h\right)} \sum_{j=i+1}^{\lfloor t\rfloor+1} \mathbb{E}\left[\left(\left|U_j\right|^2-\mathcal{V}\left(f_h\right)\right) \mid \mathcal{F}_i\right] \\
= & \frac{1}{\mathcal{V}\left(f_h\right)} \int_0^{\lfloor t\rfloor+1-i}\left(\mathbb{E} \int_{\mathbb{R}_0^d}\left|f_h\left(\mathbf{X}_s^{\mathbf{X}_i^{\mathbf{x}}}+ \mathbf{z}\right)-f_h\left(\mathbf{X}_s^{\mathbf{X}_i^{\mathbf{x}}}\right)\right|^2 \nu(\mathrm{d} \mathbf{z})-\mathcal{V}\left(f_h\right)\right) \mathrm{d} s \\
\leqslant & C\int_0^{\lfloor t\rfloor+1}\left(1+\left|\mathbf{X}_i^{\mathbf{x}}\right|^{2 \beta}\right) \mathrm{e}^{-c s} \mathrm{d} s \leqslant C\left(1+\left|\mathbf{X}_i^{\mathbf{x}}\right|^{2 \beta}\right).
\end{align*}
Hence, we can obtain
\begin{align}
	\label{e:J2}
J_2 & =\frac{2}{t^2} \sum_{i=1}^{\lfloor t\rfloor} \sum_{j=i+1}^{\lfloor t\rfloor+1} \mathbb{E}\left\{\left(\frac{1}{\mathcal{V}\left(f_h\right)}\left|U_i\right|^2-1\right) \mathbb{E}\left[\left.\left(\frac{1}{\mathcal{V}\left(f_h\right)}\left|U_j\right|^2-1\right) \right\rvert\, \mathcal{F}_i\right]\right\} \\
& \leqslant \frac{C}{t^2} \sum_{i=1}^{\lfloor t\rfloor} \mathbb{E}\left\{ \left|\frac{1}{\mathcal{V}\left(f_h\right)} | U_i|^2-1 \right| \left(1+\left|\mathbf{X}_i^{\mathbf{x}}\right|^{2 \beta}\right)\right\}\nonumber \\
& \leqslant \frac{C}{t^2} \sum_{i=1}^{\lfloor t\rfloor}\left[ \mathbb{E}\left|\frac{1}{\mathcal{V}\left(f_h\right)} | U_i |^2- 1\right|^2\right]^{1 / 2}\left[1+\mathbb{E}\left|\mathbf{X}_i^{\mathbf{x}}\right|^{4 \beta}\right]^{1 / 2} \rightarrow 0, \quad \text { as } t \rightarrow \infty.\nonumber
\end{align}
Combining \eqref{e:J1} and \eqref{e:J2} yields \eqref{condit2}, and the proof is complete.
\end{proof}

\subsection{The proof of Theorem \ref{thm:Limit2}}{\label{ProofResults2}}
Before providing the proof of the main theorem, we first present the following important non-CLT for SDEs driven by rotationally invariant $\alpha$-stable Lévy processes.
\begin{proposition}\label{thm:Limit}
Consider the solution $(\mathbf X_{t}^{\mathbf{x}})_{t \ge 0}$ of the following SDE  driven by a rotationally invariant $\alpha$-stable L\'{e}vy process with initial point $\mathbf{x}\in\mathbb{R}^{d}$,
\begin{equation}\label{SDE1}
	\mathrm{d}\mathbf X_{t}^{\mathbf{x}}=\mathbf{g}\left(\mathbf X_{t}^{\mathbf{x}}\right) \mathrm{d} t+\mathrm{d} \mathbf Z_{t},\ \mathbf X_{0}=\mathbf{x}\in\mathbb{R}^{d}.
\end{equation}
Let $\mathbf{g}: \mathbb{R}^d \rightarrow \mathbb{R}^d$ be such that the solution $(\mathbf X^{\mathbf{x}}_{t})_{t\ge0}$ of SDE \eqref{SDE1} satisfies the conditions for ergodicity and $\pi(\mathbf{b})<\infty$, where $\pi$ is the invariant measure of the SDE. Then 
\begin{equation*}
	t^{-\frac{1}{\alpha}}\int_{0}^{t}[\mathbf{g}(\mathbf X_{s}^{\mathbf{x}})-\pi(\mathbf{g})]\mathrm{d}s\Rightarrow \mathbf Z_{1}, \text{ as }t\to\infty,
\end{equation*} 
where $\mathbf Z_1$ is the random variable of the process $\mathbf Z_t$ at time $t=1$.
\end{proposition}
\begin{proof}
	Consider the case that the SDE \eqref{SDE1} starting from random initial point $\mathbf y_{0}$, that is, 
	\begin{equation}\label{StationarySDE}
		\mathrm{d}\mathbf X_{t}^{\mathbf y_{0}}=\mathbf{g}(\mathbf X_{t}^{\mathbf y_{0}})\mathrm{d}t+\mathrm{d}\mathbf Z_{t}.
	\end{equation}
 Then, by the fact that $\mathbf Z_{t}$ is rotationally symmetric, it is easy to know that,
    \begin{align*}
    \mathbb{E}\mathbf X_{t}^{\mathbf y_{0}}  = \mathbf y_{0}+\mathbb{E}\int_{0}^{t}\mathbf{g}(\mathbf X_{s}^{\mathbf y_{0}})\mathrm ds+\mathbb{E}\mathbf Z_{t} 
		 = \mathbf y_{0}+\int_{0}^{t} \mathbb{E}[\mathbf{g}(\mathbf X_{s}^{\mathbf y_0})]\mathrm ds,
	\end{align*}
	Let the initial point $\mathbf y_{0}$ obey the invariant measure $\pi$ of the SDE \eqref{SDE1}. Denote $m_{\pi}:=\mathbb{E}\mathbf X_{t}^{\mathbf y_{0}}$. We  have
    \begin{align*}  
     m_{\pi}:=\mathbb{E}\mathbf X_{t}^{\mathbf y_{0}}  = m_{\pi}+\int_{0}^{t}\int_{\mathbb R^d} \mathbb{E}[\mathbf{g}(\mathbf X_{s}^{\mathbf y})] \pi(\mathrm d\mathbf y)\mathrm ds,\quad \mathbf y_{0}\sim\pi.
	\end{align*}
     Combining SDE \eqref{SDE1} , we have
	\begin{align*}
		t^{-\frac{1}{\alpha}}\left(\mathbf X_{t}^{\mathbf{x}}-m_{\pi}\right)&=t^{-\frac{1}{\alpha}}\left(\mathbf{x}-m_{\pi}+\int_{0}^{t}\mathbf{g}(\mathbf X_{s}^{\mathbf{x}})ds-\int_{0}^{t}\int_{\mathbb R^d} \mathbb{E}[\mathbf{g}(\mathbf X_{s}^{\mathbf y})] \pi(\mathrm d\mathbf y)\mathrm ds+\mathbf Z_{t}\right)\\
		&=t^{-\frac{1}{\alpha}}(\mathbf{x}-m_{\pi})+t^{-\frac{1}{\alpha}}\int_{0}^{t}\left(\mathbf{g}(\mathbf X_{s}^{\mathbf{x}})-\int_{\mathbb R^d} \mathbb{E}[\mathbf{g}(\mathbf X_{s}^{\mathbf y})] \pi(\mathrm d\mathbf y)\right)\mathrm ds+t^{-\frac{1}{\alpha}}\mathbf Z_{t},
	\end{align*}
	which can be written as,
    \begin{align}\label{SDE2}
		 t^{-\frac{1}{\alpha}}\left(\mathbf X_{t}^{\mathbf{x}}- \mathbf{x}  - \mathbf Z_{t} \right)
		 = t^{-\frac{1}{\alpha}}\int_{0}^{t}\left(\mathbf{g}(\mathbf X_{s}^{\mathbf{x}})-\int_{\mathbb R^d} \mathbb{E}[\mathbf{g}(\mathbf X_{s}^{\mathbf y})] \pi(\mathrm d\mathbf y)\right)\mathrm ds.
	\end{align}

For the left-hand side of the equation \eqref{SDE2}, since $(\mathbf X_{t}^{\mathbf{x}})_{t\ge0}$ solves the SDE \eqref{SDE1} and $\pi$ is the unique invariant measure of the same SDE, it follows that $\mathbf X^{\mathbf{x}}_{t}\Rightarrow \mathbf Y$ for some $\mathbf Y\sim\pi$ as $t\to\infty$. Noting $\mathbf{x}$ is independent of $t$, we observe that 
\begin{align*}
t^{-\frac{1}{\alpha}}\left(\mathbf X_{t}^{x}-\mathbf{x}) \right)\Rightarrow0.
\end{align*}
 Meanwhile, since $(\mathbf Z_{t})_{t\ge0}$ is $\alpha$-stable, it follows that $t^{-\frac{1}{\alpha}}\mathbf Z_{t}\stackrel{d}{\equiv}\mathbf Z_{1}$. Combining the rotationally invariance of $(\mathbf Z_{t})_{t\ge0}$, it holds from Slutsky's theorem that  
\begin{equation}\label{RHS1}
		t^{-\frac{1}{\alpha}}[\mathbf X_{t}^{\mathbf{x}}-\mathbf{x}]-t^{-\frac{1}{\alpha}}\mathbf Z_{t}\Rightarrow \mathbf Z_{1},\text{ as }t\to\infty.
	\end{equation}

For the right-hand side of the equation \eqref{SDE2}, we know that, 
\begin{equation}\label{RHS2}
    \int_0^t\int_{\mathbb R^d} \mathbb{E}[\mathbf{g}(\mathbf X_{s}^{\mathbf y})] \pi(\mathrm d\mathbf y) \mathrm ds=\int_0^t\pi(\mathbf{g})\mathrm ds.
    \end{equation}
Combining \eqref{SDE2}, \eqref{RHS1} and \eqref{RHS2}, we can get the desired result.
\end{proof}

We analyze the regularity of the solution to the Poisson equation \eqref{Poisson} for the case $\theta>0$ and Lipschitz continuous test functions. Precisely, we consider the following Poisson equation here: for the Lipschitz continuous function $h$,
\begin{equation}
	\label{e:Lip_Poi}
	\mathcal{A} f_{h}({\bf x}) = h({\bf x}) - \mu(h).
\end{equation}
The regularity of the solution to \eqref{e:Lip_Poi} is stated in the following lemma.

\begin{lemma}\label{regularity1}
 Consider the Poisson equation \eqref{e:Lip_Poi}.  Let $h$ be a Lipschitz continuous function and $f_h$ be the solution to \eqref{e:Lip_Poi}. Then, for the case of $\theta>0$, it holds that for some $C>0$ , 
\begin{align*}
|f_h(\mathbf{x})| \leqslant C(1+|\mathbf{x}|^2)^{\frac{1-\theta}{2}} + C , \quad
\|\nabla f_h\|_\infty\leqslant C.
\end{align*}
\end{lemma}
\begin{proof}
(i)  
Since $h$ is a Lipschitz continuous function, we denote $ \| h \|_{{\rm Lip}}$ by its Lipschitz constant. Then, we have
\[
	| h({\bf x}) | \leqslant \| h \|_{{\rm Lip}} | {\bf x} | + | h({\bf 0})| , \quad \forall {\bf x} \in \mathbb{R}^d.
\]
Thus, $ | h({\bf x}) | \leq C V_{1}({\bf x})$ for some positive constant $C$, where $V_{1}({\bf x})$ is defined in \eqref{LyapunovFunction} with $\beta=1$ therein. Moreover, Lemma \ref{LemmaErgodicity2} implies that, $\mu(h) < C \mu(V_1)<\infty$, 
\begin{align*}
\left|\mathbb{E}\left[h\left(\mathbf X_t^\mathbf{x}\right)\right]-\mu(h)\right| \leq c_1 \| h \|_{{\rm Lip}} \mathrm{e}^{-c_2 t}, \quad t>1,
\end{align*}
and
\begin{align*}
\int_0^1\left|\mathbb{E}\left[h\left(\mathbf X_t^\mathbf{x}\right)\right]\right|\mathrm{d} t
\leq  C \int_0^1\mathbb{E}[ V_1\left(\mathbf X_t^\mathbf{x}\right) ]  \mathrm{d} t  
 \leq C(1+|\mathbf{x}|^2)^{\frac{1-\theta}{2}}+C.
\end{align*}
 Combining above inequalities yields that
\begin{align}{\label{Lip1}}
\left|\int_0^{\infty}\left[\mathrm{P}_t h(\mathbf{x})-\mu(h)\right] \mathrm{d} t\right| &\leq C\int_0^{\infty} \left|\mathrm{P}_t h(\mathbf{x})-\mu(h)\right| \mathrm{d} t \nonumber \\
&\leq \int_0^{1} \left[\left|\mathrm{P}_t h(\mathbf{x})\right|+\left|\mu(h)\right|\right] \mathrm{d} t+\int_1^{ \infty} c_1 \| h \|_{{\rm Lip}}  \mathrm{e}^{-c_2 t} \mathrm{d} t \\
&\leq C(1+|\mathbf{x}|^2)^{\frac{1-\theta}{2}}+C.\nonumber
\end{align} 
By Lemma \ref{lem:reghf}, we know the solution $f_h$ to \eqref{e:Lip_Poi} can be expressed as
\begin{align*}
f_h(\mathbf{x})=-\int_0^{\infty}  \mathrm{P}_t[h(\mathbf{x})-\mu(h)] \mathrm{d} t .
\end{align*}
 Thus, \eqref{Lip1} immediately yields that
$$
|f_h(\mathbf{x})| \leqslant \int_0^{\infty}\left|\mathrm{P}_t[h(\mathbf{x})-\mu(h)]\right| \mathrm{d} t \leqslant C(1+|\mathbf{x}|^2)^{\frac{1-\theta}{2}}+C.
$$ 

(ii) By \eqref{Kantorovich} and Lemma \ref{contract} (ii), we have
\begin{align*} 
\left|\mathbb{E}h\left(\mathbf X_t^{\mathbf{x}}\right)-\mathbb{E} h\left(\mathbf X_t^{\mathbf y}\right)\right|\leq \| h \|_{{\rm Lip}} \mathbb W_1\left(\mathcal{L}\left(\mathbf X_t^\mathbf{x}\right), \mathcal{L}\left(\mathbf X_t^\mathbf y\right)\right) \leq C \mathrm{e}^{-\lambda t } |\mathbf{x}-\mathbf y|,
\end{align*}
 which means that, for any $\mathbf u\in \mathbb{R}^d$,
\begin{align*}
|\nabla_\mathbf u f_h(\mathbf{x})|&=\left|\int_0^{\infty} \nabla_\mathbf u \mathbb{E}\left[h\left(\mathbf X_t^\mathbf{x}\right)-\mu(h)\right] \mathrm{d} t\right|\\
&\leq \int_0^{\infty} \lim_{\epsilon\rightarrow 0}\frac{1}{\epsilon}\left|\mathbb{E}h\left(\mathbf X_t^{\mathbf{x}+\epsilon \mathbf u}\right)-\mathbb{E} h\left(\mathbf X_t^{\mathbf{x}}\right)\right|\mathrm{d} t\\
&\leq C \int_0^{\infty} \mathrm{e}^{-\lambda t}|\mathbf u|\mathrm{d} t\leq C  |\mathbf u|.
\end{align*} 
So, we can get the desired result.
\end{proof}

\begin{proof}[Proof of Theorem \ref{thm:Limit2}]
(i) Without loss of generality, let us assume $d=1$ and $h(x)=x$, by Lemma \ref{LemmaErgodicity1}(i) and Proposition \ref{thm:Limit}, we have,
\begin{equation*}
	t^{-\frac{1}{\alpha}}\int_{0}^{t}[X_{s}^{x}-m_{\mu}]\mathrm{d}s\Rightarrow Z_{1}, \text{ as }t\to\infty,
\end{equation*} 
where $m_{\mu}:=\mathbb{E}X_{t}^{y_{0}}$, the initial point $y_{0}$ obeys the invariant measure ${\mu}$ of the SDE \eqref{SDE1} for the case $\theta=0$.
Thus, the term $\sqrt{t}\left[\mathcal{E}_t^{\mathbf{x}}(h)-\mu(h)\right]$ does not converge weakly to a Gaussian distribution, that is, the CLT does not hold.

(ii) Recalling \eqref{I1I2}, we have
\begin{align}\label{I1I2_Lip}
	& \sqrt{t}\left[ \mathcal{E}_t^{\mathbf{x}}(h)  -\mu(h)\right]=\frac{1}{\sqrt{t}} \int_0^t \mathcal{A} f_h\left(\mathbf X_s^\mathbf{x}\right) \mathrm{d}s\\
	= & \frac{1}{\sqrt{t}}\left[f_h\left(\mathbf X_t^\mathbf{x}\right)-f_h(\mathbf{x})\right]-\frac{1}{\sqrt{t}} \int_0^t \int_{\mathbb{R}^d_0}\left[f_h\left(\mathbf X_{s-}^\mathbf{x}+\mathbf z\right)-f_h\left(\mathbf X_{s-}^\mathbf{x}\right)\right] \widetilde{N}(\mathrm{d} \mathbf z, \mathrm{d} s)\nonumber\\
	=&:I_1+I_2 .\nonumber
\end{align}
The key distinction is that $h$ is a Lipschitz continuous function in this case. 

For $I_1$,  we know $|f_h(\mathbf{x})| \leqslant C(1+|\mathbf{x}|^2)^{\frac{1-\theta}{2}} + C$ by Lemma \ref{regularity1}, and $\sup _{t\geq 1}\mathbb{E}\left|\mathbf X_t^\mathbf{x}\right|^\beta \leq C$ with $1 \leq \beta<\alpha$ by Lemma \ref{LemmaErgodicity2}. Since $\theta>1-\alpha/2$ in our settings, we have that, 
\begin{align}{\label{I1}}
\mathbb E\left|\frac{1}{\sqrt{t}}\left[f_h\left(\mathbf X_t^\mathbf{x}\right)-f_h(\mathbf{x})\right]\right|&\leq \frac{1}{\sqrt{t}}\mathbb E\left[\left| f_h\left(\mathbf X_t^\mathbf{x}\right)\right|+\left|f_h(\mathbf{x}) \right|\right]\\
&\leq \frac{1}{\sqrt{t}}  \left(C(1+|\mathbf{x}|^2)^{\frac{1-\theta}{2}} + C \right) \rightarrow 0, \quad \text { as } t \rightarrow \infty.\nonumber
\end{align}

For $I_2$, by Lemma \ref{regularity1}, for all $\mathbf{x} \in \mathbb{R}^d$,
\begin{align}{\label{fh2}}
&G_{f_h}(\mathbf{x}):=\int_{\mathbb{R}_0^d}[f_h(\mathbf{x}+\mathbf z)-f_h(\mathbf{x})]^2 \nu(\mathrm{d} \mathbf z) \\
= & \int_{ |\mathbf z| \leqslant 1}\left[\int_0^1\langle\nabla f_h(\mathbf{x}+r \mathbf z), \mathbf z\rangle \mathrm{d} r\right]^2 \nu(\mathrm{d}\mathbf  z)+\int_{|\mathbf z|>1}[f_h(\mathbf{x}+\mathbf z)-f_h(\mathbf{x})]^2 \nu(\mathrm{d} \mathbf z) \nonumber\\
\leqslant & \int_{ |\mathbf z| \leqslant 1}\|\nabla f_h\|^2_{\infty}|\mathbf z|^2 \nu(\mathrm{d} \mathbf z)+2\int_{|\mathbf z|>1}[|f_h(\mathbf{x}+\mathbf z)|^2+|f_h(\mathbf{x})|^2] \nu(\mathrm{d} \mathbf z) \nonumber\\
\leqslant & C\int_{ |\mathbf z| \leqslant 1}|\mathbf z|^2 \nu(\mathrm{d} \mathbf z)+C \int_{|\mathbf z|>1}\left[ \left(1+|\mathbf{x}|^2\right) ^{1-\theta}+\left(1+|\mathbf{z}|^2\right) ^{1-\theta} + 1  \right] \nu(\mathrm{d} \mathbf z).\nonumber
\end{align}

 Hence, we can obtain that
 \begin{align}{\label{Vfh1}}
&\mathcal{V}\left(f_h\right):=\int_{\mathbb{R}^d} \int_{\mathbb{R}_0^d}[f_h(\mathbf{x}+\mathbf z)-f_h(\mathbf{x})]^2 \nu(\mathrm{d}\mathbf z) \mu(\mathrm{d} \mathbf{x})\\
&\leqslant C\int_{\mathbb{R}^d} \left[ \int_{|\mathbf z| \leqslant 1}|\mathbf z|^2 \nu(\mathrm{d} \mathbf z)+  \int_{|\mathbf z|>1}\left[\left(1+|\mathbf{x}|^2\right) ^{1-\theta}+\left(1+|\mathbf{z}|^2\right) ^{1-\theta}\right] \nu(\mathrm{d} \mathbf z) + 1 \right]\mu(\mathrm{d} \mathbf{x})\nonumber\\
&\leqslant C \mu\left( 1 + \left(1+|\mathbf{x}|^2\right) ^{1-\theta}\right)<\infty,\nonumber
\end{align} 
where the second inequality follows from the fact that $2(1-\theta)<\alpha$ and $\nu(\mathrm{d} \mathbf{z})=C_{d, \alpha}|\mathbf{z}|^{-\alpha-d} \mathrm{d} \mathbf{z}$, and the final inequality arises from $\sup _{t \geq 1} \mathbb{E} V_\beta\left(\mathbf{X}_t^{\mathbf{x}}\right) \leq C$ for $1\leq \beta<\alpha$ in Lemma \ref{LemmaErgodicity2}. 

 Applying the notations \eqref{e:U} in the proof of Theorem \ref{thm:Limit1}, we let 
\begin{align*}
U_i & =\int_{i-1}^i \int_{\mathbb{R}_0^d}\left[f_h\left(\mathbf{X}_{s-}^{\mathbf{x}}+\mathbf{z}\right)-f_h\left(\mathbf{X}_{s-}^{\mathbf{x}}\right)\right] \widetilde{N}( \mathrm{d} \mathbf{z},\mathrm{d} s), \quad \text { for } i=1,2, \cdots,\lfloor t\rfloor, \\
U_{\lfloor t\rfloor+1} & =\int_{\lfloor t\rfloor}^t \int_{\mathbb{R}_0^d}\left[f_h\left(\mathbf{X}_{s-}^{\mathbf{x}}+\mathbf{z}\right)-f_h\left(\mathbf{X}_{s-}^{\mathbf{x}}\right)\right] \widetilde{N}(\mathrm{d} \mathbf{z}, \mathrm{d} s).
\end{align*}

Since $\{U_i\}$ are martingale differences, by Lemma \ref{LemmaErgodicity2}, the fact that $2(1-\theta) < \alpha$ and the calculation which is similar with inequality \eqref{Vfh1} for all $i=1,2, \cdots,\lfloor t\rfloor+1$, we have
 \begin{align}\label{UU2}
    \mathbb{E} |U_i|^2&=\int_{i-1}^i \int_{\mathbb{R}_0^d}\mathbb{E} \left[f_h\left(\mathbf{X}_{s }^{\mathbf{x}}+\mathbf{z}\right)-f_h\left(\mathbf{X}_{s }^{\mathbf{x}}\right)\right]^2 \nu(\mathrm{d} \mathbf z)\mathrm{d} s\\
    &\leqslant C \int_{i-1}^i\left[ \mathbb{E}\left(1+\left|\mathbf{X}_{s}^{\mathbf{x}}\right|^2\right)^{1-\theta} + 1 \right]\mathrm{d} s \leqslant C \left[ \left(1+|\mathbf{x}|^{2}\right)^{\frac{\alpha - \theta}{2}} + 1\right]<\infty.\nonumber
\end{align} 
Similar to the proof of Theorem \ref{thm:Limit1}, we also claim that
\begin{align}
&\lim _{t \rightarrow \infty} \mathbb{E}\left[\max _{1 \leqslant i \leqslant\lfloor t\rfloor+1} \frac{1}{\sqrt{\mathcal{V}\left(f_h\right) t}}\left|U_i\right|\right]=0,\label{U5} \\
&\sum_{i=1}^{\lfloor t\rfloor+1} \frac{1}{\mathcal{V}\left(f_h\right) t}\left|U_i\right|^2 \xrightarrow{\mathrm{P}} 1, \quad \text { as } t \rightarrow \infty .\label{U6}
\end{align}

The proof of \eqref{U5} is omitted, as it follows a similar approach to the proof of \eqref{UU5}. Furthermore, we can verify \eqref{U6} by proving
\begin{equation}\label{e:U7}
	\lim _{t \rightarrow \infty} \mathbb{E}\left|\sum_{i=1}^{\lfloor t\rfloor+1} \frac{1}{\mathcal{V}\left(f_h\right) t}| U_i|^2-\frac{\lfloor t\rfloor+1}{t}\right|^2=0.
\end{equation}
Define the set $A_i=\left\{\left|U_i\right|^2 \leq t^\epsilon\right\}$, and let $A_i^c$ represent its complement, where $\epsilon>0$ is sufficiently small. Then, the right-hand side of \eqref{e:U7} satisfies that 
\begin{align*}
  &\mathrel{\phantom{=}}
  \mathbb{E}\left|\sum_{i=1}^{\lfloor t\rfloor+1} \frac{1}{\mathcal{V}\left(f_h\right) t}| U_i|^2-\frac{\lfloor t\rfloor+1}{t}\right| \\
&\leq \mathbb{E}\left[\left|\sum_{i=1}^{\lfloor t\rfloor+1} \frac{1}{\mathcal{V}\left(f_h\right) t}| U_i|^2 \mathbf 1_{A_i} -\frac{\lfloor t\rfloor+1}{t}\right|\right] +\mathbb{E}\left[\left|\sum_{i=1}^{\lfloor t\rfloor+1} \frac{1}{\mathcal{V}\left(f_h\right) t}| U_i|^2 \mathbf 1_{A_i^c}  \right|\right]\\
&=:  J_1+J_2.
\end{align*}

As for the term $J_1$, observe that, 
\begin{align*}
&\mathrel{\phantom{=}} \mathbb{E}\left[\left|\sum_{i=1}^{\lfloor t\rfloor+1} \frac{1}{\mathcal{V}\left(f_h\right) t} | U_i |^2\mathbf 1_{A_i}- \frac{\lfloor t\rfloor+1}{t}\right|^2 \right] 
= \mathbb{E}\left[\left|\frac{1}{t} \sum_{i=1}^{\lfloor t\rfloor+1}\left(\frac{1}{\mathcal{V}\left(f_h\right)}\left|U_i\right|^2\mathbf 1_{A_i}-1\right)\right|^2  \right]\\
&=  \frac{1}{t^2} \sum_{i=1}^{\lfloor t\rfloor+1} \mathbb{E}\left[\left(\frac{1}{\mathcal{V}\left(f_h\right)}\left|U_i\right|^2\mathbf 1_{A_i}-1\right)^2 \right]\\
&\mathrel{\phantom{=}} +\frac{2}{t^2} \sum_{i<j} \mathbb{E}\left[\left(\frac{1}{\mathcal{V}\left(f_h\right)}\left|U_i\right|^2\mathbf 1_{A_i}-1\right)\left(\frac{1}{\mathcal{V}\left(f_h\right)}\left|U_j\right|^2\mathbf 1_{A_j}-1\right) \right]
=:  J_{11} +  J_{12} .
\end{align*}

As for the term $J_{11}$, by \eqref{UU2}, $\mathbb{E} |U_i|^2\leq C \big[ \left(1+|\mathbf{x}|^{2}\right)^{\frac{\alpha-\theta}{2}} + 1 \big]$. Then, there exists some constant $C>0$ such that
$$
\mathbb{E}\left[\left|U_i\right|^4 \mathbf 1_{A_i}\right]\leqslant (\mathbb{E} |U_i|^2)t^{\epsilon}  \leqslant C \left[ \left(1+|\mathbf{x}|^{2}\right)^{\frac{\alpha-\theta}{2}} + 1 \right] t^{\epsilon}.
$$
So, it is easy to know that,
\begin{align*}
J_{11}=&\frac{1}{t^2} \sum_{i=1}^{\lfloor t\rfloor+1} \mathbb{E}\left[\left(\frac{1}{\mathcal{V}\left(f_h\right)}\left|U_i\right|^2\mathbf 1_{A_i}-1\right)^2\right]\\
=&\frac{1}{t^2} \sum_{i=1}^{\lfloor t\rfloor+1} \mathbb{E}\left[\left(\frac{1}{\mathcal{V}\left(f_h\right)^2}\left|U_i\right|^4\mathbf 1_{A_i}-\frac{2}{\mathcal{V}\left(f_h\right)}\left|U_i\right|^2\mathbf 1_{A_i}+1\right)\right]\\
\leqslant &\frac{1}{t^2} \sum_{i=1}^{\lfloor t\rfloor+1} \left[C\left(\left(1+|\mathbf{x}|^{2}\right)^{\frac{\alpha-\theta}{2}}+1\right) t^{\epsilon}+C\right]\\
\leqslant & C\left(\left(1+|\mathbf{x}|^{2}\right)^{\frac{\alpha-\theta}{2}}+1\right) t^{\epsilon-1}+Ct^{-1}\rightarrow 0, \quad \text { as } t \rightarrow \infty.
\end{align*}

As for the term $J_{12}$. By \eqref{fh2} and the fact that $2(1-\theta)<\alpha$, we know that there exist  a constant $\beta\in [1,\alpha)$ and $C>0$ such that
$$
|G_{f_h}(\mathbf{x})|\leq C  \left( 1 + \left(1+|\mathbf{x}|^2\right) ^{1-\theta}\right)\leq C V_\beta(\mathbf{x}).
$$
Since $\{U_i\}$ are martingale differences, fixed $i=1, \ldots,\lfloor t\rfloor$, by the fact that $\mathcal V\left(f_h\right):=\mu( G_{f_h})$ and \eqref{LemmaErgodicity21} in Lemma \ref{LemmaErgodicity0}  , we observe that for any $j>i$,
\begin{align}{\label{conditionF}}
&\mathbb{E}\left[\left.\left(\frac{1}{\mathcal V\left(f_h\right)}\left|U_j\right|^2-1\right) \right\rvert\, \mathcal{F}_{i}\right] \\
=&\frac{1}{\mathcal V\left(f_h\right)} \int_{j-i-1}^{j-i} \int_{\mathbb{R}_0^{d}} \mathbb{E}\left[f_h\left(\mathbf X_{s}^{\mathbf X_{i}^\mathbf{x}}+\mathbf z\right)-f_h\left(\mathbf X_{s}^{\mathbf X_{i}^\mathbf{x}}\right)\right]^2 \nu(\mathrm{d} \mathbf z) \mathrm{d} s-1\nonumber \\
 =&\frac{1}{\mathcal V\left(f_h\right)} \int_{j-i-1}^{j-i}\left[ \int_{\mathbb{R}_0^d}\mathbb{E}\left[f_h\left(\mathbf X_{s}^{\mathbf X_{i}^\mathbf{x}}+\mathbf z\right)-f_h\left(\mathbf X_{s}^{\mathbf X_{i}^\mathbf{x}}\right)\right]^2 \nu(\mathrm{d} \mathbf z)-\mathcal V\left(f_h\right)\right] \mathrm{d} s\nonumber \\
 =&\frac{1}{\mathcal V\left(f_h\right)} \int_{j-i-1}^{j-i}\mathbb E\left[ G_{f_h}(\mathbf X_{s}^{\mathbf X_{i}^\mathbf{x}})-\mu( G_{f_h})\right] \mathrm{d} s\nonumber\\
\leq &\frac{C}{\mathcal V\left(f_h\right)} \int_{j-i-1}^{j-i}\sup _{|f| \leq V_\beta}\left|\mathbb{E}\left[f\left(\mathbf X_{s}^{\mathbf X_{i}^\mathbf{x}}\right)\right]-\mu(f)\right| \mathrm{d} s\nonumber\\
 \leqslant& { \frac{C}{\mathcal V\left(f_h\right)} \int_{j-1}^{j}  \left(1+\left|\mathbf{X}_i^{\mathbf{x}}\right|^2\right)^{{\beta}/{2}} \mathrm{e}^{-c s} \mathrm{d} s.}\nonumber
\end{align}
As we know $J_{12}$ can be rewritten as
\begin{align}{\label{J120}}
&J_{12} :=\frac{2}{t^2} \sum_{i<j} \mathbb{E}\left[\left(\frac{1}{\mathcal{V}\left(f_h\right)}\left|U_i\right|^2\mathbf 1_{A_i}-1\right)\left(\frac{1}{\mathcal{V}\left(f_h\right)}\left|U_j\right|^2\mathbf 1_{A_j}-1\right) \right]\\
=&\frac{2}{t^2} \sum_{i<j} \mathbb{E}\left[\left(\frac{1}{\mathcal{V}\left(f_h\right)}\left|U_i\right|^2\mathbf 1_{A_i}-1\right)\left(\frac{1}{\mathcal{V}\left(f_h\right)}\left|U_j\right|^2 -1\right) \right]\nonumber\\
& -\frac{2}{t^2} \sum_{i<j} \mathbb{E}\left[\left(\frac{1}{\mathcal{V}\left(f_h\right)}\left|U_i\right|^2\mathbf 1_{A_i}-1\right)\left(\frac{1}{\mathcal{V}\left(f_h\right)}\left|U_j\right|^2\mathbf 1_{A^c_j}\right) \right]\nonumber\\
=&\frac{2}{t^2} \sum_{i<j} \mathbb{E}\left[\left(\frac{1}{\mathcal{V}\left(f_h\right)}\left|U_i\right|^2\mathbf 1_{A_i}-1\right)\mathbb{E}\left[\left(\frac{1}{\mathcal{V}\left(f_h\right)}\left|U_j\right|^2 -1\right)\rvert\, \mathcal{F}_{i } \right]\right]\nonumber\\
& -\frac{2}{t^2\mathcal{V}\left(f_h\right)^2} \sum_{i<j} \mathbb{E}\left[\left( \left|U_i\right|^2\mathbf 1_{A_i} \right)\left( \left|U_j\right|^2\mathbf 1_{A^c_j}\right) \right]+\frac{2}{t^2} \sum_{i<j} \mathbb{E}\left[\frac{1}{\mathcal{V}\left(f_h\right)}\left|U_j\right|^2\mathbf 1_{A^c_j}\right].\nonumber
\end{align}
Hence, combining \eqref{conditionF} and \eqref{J120}, we can obtain that, there exist constant $1\leq \beta<\alpha$ and $C>0$ such that,  
\begin{align}{\label{J12}}
J_{12}&\leq \frac{2}{\mathcal{V}\left(f_h\right) t^2} \sum_{i=1}^{\lfloor t\rfloor} \mathbb{E}\left\{ \left|\frac{1}{\mathcal{V}\left(f_h\right)} | U_i |^2\mathbf 1_{A_i}-1 \right\rvert\, \int_i^t C \left(1+\left|\mathbf{X}_i^{\mathbf{x}}\right|^2\right)^{ {\beta}/{2}} {\rm e}^{-c s} \mathrm{ d} s\right\}\\
 &\quad+\frac{2}{\mathcal{V}\left(f_h\right)t^2} \sum_{i=1}^{\lfloor t\rfloor} \sum_{j=i+1}^{\lfloor t\rfloor+1}\mathbb{E}\left[   \left|U_j\right|^2\mathbf 1_{A^c_j} \right]\nonumber\\
& \leqslant \frac{C}{t^2} \sum_{i=1}^{\lfloor t\rfloor} (1+t^{\epsilon})  \mathbb{E}\left[\left(1+\left|\mathbf{X}_i^{\mathbf{x}}\right|^2\right)^{{\beta}/{2}} \right]+\frac{C}{t }  \sum_{j=1}^{\lfloor t\rfloor+1}   \mathbb{E}\left[  \left|U_j\right|^2\mathbf 1_{A^c_j} \right].\nonumber
\end{align}

For the second term of the right side of \eqref{J12}. As we know $\epsilon>0$ is sufficiently small, which means that there exists $\epsilon>0$ such that $2(1-\theta)<\frac{2(1-\theta)}{1-\epsilon}<\alpha$. By using Kunita's inequality (see, \cite[Theorem 4.4.23]{applebaum2009levy}) and Lemma \ref{LemmaErgodicity2}, there exists some positive constant $C$ such that
\begin{align}{\label{Uepsilon}}
\mathbb{E}\left[\left|U_i\right|^{\frac{2}{1-\epsilon}}\right] \leqslant &   \mathbb{E}\left|\int_{i-1}^i \int_{\mathbb{R}_0^d}\left[f_h\left(\mathbf{X}_{s}^{\mathbf{x}}+ \mathbf{z}\right)-f_h\left(\mathbf{X}_{s}^{\mathbf{x}}\right)\right]^2 \nu(\mathrm{d} \mathbf{z}) \mathrm{d} s\right|^{\frac{1}{1-\epsilon}}  \\
& +\mathbb{E} \int_{i-1}^i \int_{\mathbb{R}_0^d}\left[f_h\left(\mathbf{X}_{s}^{\mathbf{x}}+  \mathbf{z}\right)-f_h\left(\mathbf{X}_{s}^{\mathbf{x}}\right)\right]^{\frac{2}{1-\epsilon}} \nu(\mathrm{d} \mathbf{z}) \mathrm{d} s\nonumber\\
\leqslant& C \int_{i-1}^i\left[\mathbb{E}\left(1+\left|\mathbf{X}_{s}^{\mathbf{x}}\right|^2\right)^{\frac{1-\theta}{1-\epsilon}}+1\right]\mathrm{d} s \leqslant C\left[\left(1+|\mathbf{x}|^{2}\right)^{\frac{\alpha- \theta}{2 }}+1\right].\nonumber
\end{align}
Then, by Holder's inequality and Markov's inequality, we have
\begin{align}{\label{EU}}
\frac1t\sum_{i=1}^{\lfloor t\rfloor+1}\mathbb{E}\left(  | U_i|^2 \mathbf 1_{A_i^c}\right)
\leq& \frac1t \sum_{i=1}^{\lfloor t\rfloor+1}\left(\mathbb{E}  |U_i| ^{\frac{2}{1-\epsilon}} \right)^{1-\epsilon}\left(\mathbb P\left(  \left|U_i\right|^{2 }> t^{ \epsilon} \right)\right)^{\epsilon}  \\
\leq &   \frac{C}{t}\left[\left(1+|\mathbf{x}|^{2}\right)^{\frac{\alpha- \theta}{2}(1-\epsilon)}+1\right]\sum_{i=1}^{\lfloor t\rfloor+1}\left(\mathbb P\left(  \left|U_i\right|^{2 }> t^{ \epsilon} \right)\right)^{\epsilon}\nonumber\\
\leq &\frac{C}{t}\left[\left(1+|\mathbf{x}|^{2}\right)^{\frac{\alpha- \theta}{2}(1-\epsilon)}+1\right]\sum_{i=1}^{\lfloor t\rfloor+1}\left(\frac{\mathbb E\left|U_i\right|^{2 }}{t^{ \epsilon}} \right)^{\epsilon} \nonumber\\
\leq & Ct^{ -\epsilon^2}\left[\left(1+|\mathbf{x}|^{2}\right)^{\frac{\alpha- \theta}{2}}+1\right] .\nonumber
\end{align}

Then, combining \eqref{J12} and \eqref{EU}. By  \eqref{e:Lem2.4_1}, we have
\begin{align*}
J_{12}&\leq \frac{C}{t^2} \sum_{i=1}^{\lfloor t\rfloor} (1+t^{\epsilon})   +C t^{ -\epsilon^2}\left(1+|\mathbf{x}|^{2}\right)^{\frac{\alpha- \theta}{2}}\\
&\leq  C(t^{-1}+t^{\epsilon-1})   + {C} t^{ -\epsilon^2}\left[\left(1+|\mathbf{x}|^{2}\right)^{\frac{\alpha- \theta}{2}}+1\right]\rightarrow 0, \text{ as } t \rightarrow \infty.
\end{align*}

At last, we turn to the term $J_2$. By performing a calculation similar to that in \eqref{EU}, we obtain
\begin{align*}
J_2:=&\mathbb{E}\left[\left|\sum_{i=1}^{\lfloor t\rfloor+1} \frac{1}{\mathcal{V}\left(f_h\right) t}| U_i|^2\mathbf 1_{A_i^c}\right|\right]
= \frac{1}{t}\sum_{i=1}^{\lfloor t\rfloor+1}\mathbb{E}\left( \frac{1}{\mathcal{V}\left(f_h\right) }| U_i|^2 \mathbf 1_{A_i^c}\right)\\
\leq &Ct^{-\epsilon^2}\left[\left(1+|\mathbf{x}|^{2}\right)^{\frac{\alpha- \theta}{2}}+1\right]\rightarrow 0, \text{ as } t \rightarrow \infty.
\end{align*}
Thus, we have completed the proof. 
\end{proof}

\bibliographystyle{amsplain}
\bibliography{Existence_and_non-existence_of_the_CLT_for_a_family_of_SDEs_driven_by_stable_process}

\providecommand{\bysame}{\leavevmode\hbox to3em{\hrulefill}\thinspace}
\providecommand{\MR}{\relax\ifhmode\unskip\space\fi MR }
\providecommand{\MRhref}[2]{%
  \href{http://www.ams.org/mathscinet-getitem?mr=#1}{#2}
}
\providecommand{\href}[2]{#2}
\begin{thebibliography}{10}

\bibitem{applebaum2009levy}
David Applebaum, \emph{L{\'e}vy processes and stochastic calculus}, second ed.,
  Cambridge Studies in Advanced Mathematics, vol. 116, Cambridge University
  Press, Cambridge, 2009. \MR{2512800}

\bibitem{arapostathis2019uniform}
Ari Arapostathis, Hassan Hmedi, Guodong Pang, and Nikola Sandri\'{c},
  \emph{Uniform polynomial rates of convergence for a class of
  {L}\'{e}vy-driven controlled {SDE}s arising in multiclass many-server
  queues}, Modeling, stochastic control, optimization, and applications, IMA
  Vol. Math. Appl., vol. 164, Springer, Cham, 2019, pp.~1--20. \MR{3970163}

\bibitem{bao2024limit}
Jianhai Bao and Jiaqing Hao, \emph{Limit theorems for {SDEs} with irregular
  drifts}, arXiv preprint arXiv:2403.06192 (2024).

\bibitem{MR4128304}
Jianhai Bao, Feng-Yu Wang, and Chenggui Yuan, \emph{Limit theorems for additive
  functionals of path-dependent {SDE}s}, Discrete Contin. Dyn. Syst.
  \textbf{40} (2020), no.~9, 5173--5188. \MR{4128304}

\bibitem{bao2022coupling}
Jianhai Bao and Jian Wang, \emph{Coupling approach for exponential ergodicity
  of stochastic {H}amiltonian systems with {L}\'{e}vy noises}, Stochastic
  Process. Appl. \textbf{146} (2022), 114--142. \MR{4370566}

\bibitem{baoyuan2017}
Jianhai Bao, George Yin, and Chenggui Yuan, \emph{Two-time-scale stochastic
  partial differential equations driven by {$\alpha$}-stable noises: averaging
  principles}, Bernoulli \textbf{23} (2017), no.~1, 645--669. \MR{3556788}

\bibitem{bao2011comparison}
Jianhai Bao and Chenggui Yuan, \emph{Comparison theorem for stochastic
  differential delay equations with jumps}, Acta Appl. Math. \textbf{116}
  (2011), no.~2, 119--132. \MR{2842984}

\bibitem{bao2012stochastic}
\bysame, \emph{Stochastic population dynamics driven by {L}\'{e}vy noise}, J.
  Math. Anal. Appl. \textbf{391} (2012), no.~2, 363--375. \MR{2903137}

\bibitem{MR481057}
Christian Berg and Gunnar Forst, \emph{Potential theory on locally compact
  abelian groups}, Ergebnisse der Mathematik und ihrer Grenzgebiete [Results in
  Mathematics and Related Areas], vol. Band 87, Springer-Verlag, New
  York-Heidelberg, 1975. \MR{481057}

\bibitem{borovkov2001piece}
K.~Borovkov and A.~Novikov, \emph{On a piece-wise deterministic {M}arkov
  process model}, Statist. Probab. Lett. \textbf{53} (2001), no.~4, 421--428.
  \MR{1856167}

\bibitem{WOS:000343456100011}
Bj\"{o}rn B\"{o}ttcher, Ren\'{e} Schilling, and Jian Wang, \emph{L\'{e}vy
  matters. {III}}, Lecture Notes in Mathematics, vol. 2099, Springer, Cham,
  2013, L\'{e}vy-type processes: construction, approximation and sample path
  properties, With a short biography of Paul L\'{e}vy by Jean Jacod, L\'{e}vy
  Matters. \MR{3156646}

\bibitem{chen2016heat}
Zhen-Qing Chen and Xicheng Zhang, \emph{Heat kernels and analyticity of
  non-symmetric jump diffusion semigroups}, Probab. Theory Related Fields
  \textbf{165} (2016), no.~1-2, 267--312. \MR{3500272}

\bibitem{chen2021supercritical}
Zhen-Qing Chen, Xicheng Zhang, and Guohuan Zhao, \emph{Supercritical {SDE}s
  driven by multiplicative stable-like {L}\'evy processes}, Trans. Amer. Math.
  Soc. \textbf{374} (2021), no.~11, 7621--7655. \MR{4328678}

\bibitem{dereich2016multilevel}
Steffen Dereich and Sangmeng Li, \emph{Multilevel {M}onte {C}arlo for
  {L}\'evy-driven {SDE}s: central limit theorems for adaptive {E}uler schemes},
  Ann. Appl. Probab. \textbf{26} (2016), no.~1, 136--185. \MR{3449315}

\bibitem{dong2020irreducibility}
Zhao Dong, Feng-Yu Wang, and Lihu Xu, \emph{Irreducibility and asymptotics of
  stochastic {B}urgers equation driven by {$\alpha$}-stable processes},
  Potential Anal. \textbf{52} (2020), no.~3, 371--392. \MR{4067297}

\bibitem{durrett2019probability}
Rick Durrett, \emph{Probability---theory and examples}, fifth ed., Cambridge
  Series in Statistical and Probabilistic Mathematics, vol.~49, Cambridge
  University Press, Cambridge, 2019. \MR{3930614}

\bibitem{fang2019multivariate}
Xiao Fang, Qi-Man Shao, and Lihu Xu, \emph{Multivariate approximations in
  {W}asserstein distance by {S}tein's method and {B}ismut's formula}, Probab.
  Theory Related Fields \textbf{174} (2019), no.~3-4, 945--979. \MR{3980309}

\bibitem{MR2362589}
Brice Franke, \emph{A functional non-central limit theorem for jump-diffusions
  with periodic coefficients driven by stable {L}\'evy-noise}, J. Theoret.
  Probab. \textbf{20} (2007), no.~4, 1087--1100. \MR{2362589}

\bibitem{MR1949404}
Desmond~J. Higham, Xuerong Mao, and Andrew~M. Stuart, \emph{Strong convergence
  of {E}uler-type methods for nonlinear stochastic differential equations},
  SIAM J. Numer. Anal. \textbf{40} (2002), no.~3, 1041--1063. \MR{1949404}

\bibitem{MR2051587}
\bysame, \emph{Exponential mean-square stability of numerical solutions to
  stochastic differential equations}, LMS J. Comput. Math. \textbf{6} (2003),
  297--313. \MR{2051587}

\bibitem{hutzenthaler2011strong}
Martin Hutzenthaler, Arnulf Jentzen, and Peter~E. Kloeden, \emph{Strong and
  weak divergence in finite time of {E}uler's method for stochastic
  differential equations with non-globally {L}ipschitz continuous
  coefficients}, Proc. R. Soc. Lond. Ser. A Math. Phys. Eng. Sci. \textbf{467}
  (2011), no.~2130, 1563--1576. \MR{2795791}

\bibitem{MR3134726}
\bysame, \emph{Divergence of the multilevel {M}onte {C}arlo {E}uler method for
  nonlinear stochastic differential equations}, Ann. Appl. Probab. \textbf{23}
  (2013), no.~5, 1913--1966. \MR{3134726}

\bibitem{jin2024approximation}
Xinghu Jin, Guodong Pang, Lihu Xu, and Xin Xu, \emph{An approximation to the
  invariant measure of the limiting diffusion of {G/Ph/n + GI} queues in the
  {Halfin--Whitt} regime and related asymptotics}, Mathematics of Operations
  Research (2024).

\bibitem{KIM20142479}
Panki Kim and Renming Song, \emph{Stable process with singular drift},
  Stochastic Process. Appl. \textbf{124} (2014), no.~7, 2479--2516.
  \MR{3192504}

\bibitem{KUHN20192654}
Franziska K\"{u}hn and Ren\'{e}~L. Schilling, \emph{Strong convergence of the
  {E}uler-{M}aruyama approximation for a class of {L}\'{e}vy-driven {SDE}s},
  Stochastic Process. Appl. \textbf{129} (2019), no.~8, 2654--2680.
  \MR{3980140}

\bibitem{li2023stable}
Xiang Li, Lihu Xu, and Haoran Yang, \emph{Stable central limit theorem in total
  variation distance}, J. Theoret. Probab. \textbf{38} (2025), no.~1, Paper No.
  16, 51. \MR{4833281}

\bibitem{lu2022central}
Jianya Lu, Yuzhen Tan, and Lihu Xu, \emph{Central limit theorem and
  self-normalized {C}ram\'er-type moderate deviation for {E}uler-{M}aruyama
  scheme}, Bernoulli \textbf{28} (2022), no.~2, 937--964. \MR{4388925}

\bibitem{Luo2016RefinedBC}
Dejun Luo and Jian Wang, \emph{Refined basic couplings and {W}asserstein-type
  distances for {SDE}s with {L}{\'e}vy noises}, Stochastic Process. Appl.
  \textbf{129} (2019), no.~9, 3129--3173. \MR{3985558}

\bibitem{meyn1993stability}
Sean~P. Meyn and R.~L. Tweedie, \emph{Stability of {M}arkovian processes.
  {III}. {F}oster-{L}yapunov criteria for continuous-time processes}, Adv. in
  Appl. Probab. \textbf{25} (1993), no.~3, 518--548. \MR{1234295}

\bibitem{qiao2022limit}
Huijie Qiao, \emph{Limit theorems of {SDE}s driven by {L}\'evy processes and
  application to nonlinear filtering problems}, NoDEA Nonlinear Differential
  Equations Appl. \textbf{29} (2022), no.~1, Paper No. 8, 25. \MR{4363815}

\bibitem{MR3185174}
Ken-iti Sato, \emph{L\'evy processes and infinitely divisible distributions},
  revised ed., Cambridge Studies in Advanced Mathematics, vol.~68, Cambridge
  University Press, Cambridge, 2013, Translated from the 1990 Japanese
  original. \MR{3185174}

\bibitem{sethuraman2002martingale}
Sunder Sethuraman, \emph{A martingale central limit theorem}, PDF accessible on
  author’s website: https://www. math. arizona. edu/sethuram/notes/wi mart1.
  pdf \textbf{17} (2002).

\bibitem{MR2160585}
Rong Situ, \emph{Theory of stochastic differential equations with jumps and
  applications}, Mathematical and Analytical Techniques with Applications to
  Engineering, Springer, New York, 2005, Mathematical and analytical techniques
  with applications to engineering. \MR{2160585}

\bibitem{suo2021central}
Yongqiang Suo and Chenggui Yuan, \emph{Central limit theorem and moderate
  deviation principle for {M}c{K}ean-{V}lasov {SDE}s}, Acta Appl. Math.
  \textbf{175} (2021), Paper No. 16, 19. \MR{4327447}

\bibitem{wang2014harnack}
Feng-Yu Wang and Jian Wang, \emph{Harnack inequalities for stochastic equations
  driven by {L}\'evy noise}, J. Math. Anal. Appl. \textbf{410} (2014), no.~1,
  513--523. \MR{3109860}

\bibitem{wang2016p}
Jian Wang, \emph{{$L^p$}-{W}asserstein distance for stochastic differential
  equations driven by {L}\'evy processes}, Bernoulli \textbf{22} (2016), no.~3,
  1598--1616. \MR{3474827}

\bibitem{wang2021large}
Ran Wang, Jie Xiong, and Lihu Xu, \emph{Large deviation principle of occupation
  measures for non-linear monotone {SPDE}s}, Sci. China Math. \textbf{64}
  (2021), no.~4, 799--822. \MR{4236115}

\bibitem{wang2024phase}
Yu~Wang, Yimin Xiao, and Lihu Xu, \emph{Phase transition in the em scheme of an
  sde driven by $\alpha$-stable noises with $\alpha \in (0,2]$}, 2024.

\bibitem{xu2019approximation}
Lihu Xu, \emph{Approximation of stable law in {W}asserstein-1 distance by
  {S}tein's method}, Ann. Appl. Probab. \textbf{29} (2019), no.~1, 458--504.
  \MR{3910009}

\bibitem{zhang2013derivative}
Xicheng Zhang, \emph{Derivative formulas and gradient estimates for {SDE}s
  driven by {$\alpha$}-stable processes}, Stochastic Process. Appl.
  \textbf{123} (2013), no.~4, 1213--1228. \MR{3016221}

\end{thebibliography}

\end{document}